\documentclass[11pt]{article}
\usepackage[tbtags]{amsmath}
\usepackage{amssymb}
\usepackage{amsthm}
\usepackage[misc]{ifsym}
\usepackage{cases}
\usepackage{cite}
\usepackage{mathrsfs}
\usepackage{xcolor}
\usepackage{graphicx} 
\usepackage{float}    
\usepackage{subfig}
\usepackage{caption}
\usepackage{tikz}
\usetikzlibrary{arrows,shapes,chains}

\numberwithin{equation}{section}   
\normalsize
\title{\bf A Linear Quadratic Stochastic Stackelberg Differential Game with Time Delay
  \thanks{This work is financially supported by the National Key R\&D Program of China (2018YFB1305400), and the National Natural Science Foundations of China (11971266, 11571205, 11831010).}}
\author{\normalsize Weijun Meng\thanks{\it School of Mathematics, Shandong University, Jinan 250100, P.R. China, E-mail: mengwj@mail.sdu.edu.cn},\ Jingtao Shi\thanks{\it Corresponding author, School of Mathematics, Shandong University, Jinan 250100, P.R. China, E-mail: shijingtao@sdu.edu.cn}}

\newtheorem{mythm}{Theorem}[section]

\newtheorem{mylem}{Lemma}[section]
\newtheorem{Remark}{Remark}[section]

\begin{document}

\maketitle \noindent{\bf Abstract:}\quad
This paper is concerned with a linear quadratic stochastic Stackelberg differential game with time delay. The model is general, in which the state delay and the control delay both appear in the state equation, moreover, they both enter into the diffusion term. By introducing two Pseudo-Riccati equations and a special matrix equation, the state feedback representation of the open-loop Stackelberg strategy is derived, under some assumptions. Finally, two examples are given to illustrate the applications of the theoretical results.

\vspace{2mm}


\vspace{2mm}


\section{Introduction}

Stackelberg differential games have been investigated in order to obtain optimal strategies in competitive economics in which there exist leaders and followers. Stackelberg strategies are rational solutions when there are two players and the leader knows the rational reaction of the follower and reveals first his strategy, while the follower does not know the rational reaction of the leader and has to optimize his criterion for a given control of the leader. H. von Stackelberg \cite{Stack34} first introduced a hierarchical solution for markets with leaders and followers in 1934, which is now known as the Stackelberg equilibrium. The notion of the Stackelberg solution was later extended to a multiperiod setting, see Simaan and Cruz \cite{SC73}. Among these, for the deterministic dynamics, the Stackelberg strategy was studied in Papavassilopoulos and Cruz \cite{PC79}, Pan and Yong \cite{PY91}, Ba\c{s}ar and Olsder \cite{BO95}, Freiling et al. \cite{FJL01}. And for the stochastic dynamics when random noises are considered, the Stackelberg strategy was considered in Castanon and Athans \cite{CA76}, Ba\c{s}ar \cite{B79}, Yong \cite{Y02}, Bensoussan et al. \cite{BCS15}, Shi et al. \cite{SWX16,SWX17,SWX20} and the references therein. 

In the above results, the states of the controlled systems only depend on the value of the current time. However, the development of some random phenomena in the real world depends not only on their current value, but also on their past history. So far there have been extensive literature to study the stochastic optimal control problems with time delay, see Mohammed \cite{Mohammed84,Mohammed98}, \O ksendal and Sulem \cite{OS00}, Chen and Wu \cite{CW10}, Huang et al. \cite{HLS12}, Zhang and Xu \cite{ZX17} and the references therein. Hence it is of great importance to consider time delay in the dynamics of Stackelberg differential games. There exist a few literatures in this topic and let us mention some of them. Harband \cite{H77} studied the existence of monotonic solutions of a nonlinear car-following equation. Ishida and Shimemura \cite{IS83} derived the sufficient conditions for constructing a team-optimal closed-loop Stackelberg strategies in a linear quadratic differential game with time delay. Ishida and Shimemura \cite{IS87} researched necessary and sufficient conditions for the open-loop Stackelberg strategies in a linear quadratic differential game with time delay, where the evolution of the game with time delay is described by coupled differential equations composed of lumped and distributed parameter subsystems. \O ksendal et al. \cite{OSU14} considered explicit formulae for equilibrium prices in a continuous-time vertical contracting model, in which the information is delayed. Bensoussan et al. \cite{BCY15} investigated an $N$-player interacting strategic game in the presence of a (endogenous) dominating player, who gives direct influence on individual agents, through its impact on their control in the sense of Stackelberg game, and then on the whole community. But only the leader's state delay appears in the model. Xu and Zhang \cite{XZ16} focused on leader-follower differential games with time delay and overcame the non-causality difficulty of strategy design caused by the delay by introducing the new co-states which capture the future information of the control and the new state which contains the past effects, however the system is deterministic. Bensoussan et al. \cite{BCLY17} was concerned with a linear-quadratic mean-field game between a leader and a group of followers, but no delay appears in the diffusion term. Xu et al. \cite{XSZ18} obtained the open-loop Stackelberg strategy for the linear quadratic leader-follower stochastic differential game with time delay only appearing in the leader's problem. Li et al. \cite{LWXZ18} considered a Stackelberg strategy for deterministic mixed $H_2/H_\infty$ control problem with time delay in the leader's control. Bai et al. \cite{BZXG20} addressed a Stackelberg stochastic differential reinsurance investment game problem and derived the equilibrium strategy explicitly for the game, however the diffusion term of the follower does not contain the time delay.

As mentioned above, in the above literatures, either the delays do not both appear in the state equations of leader and follower, or the state delay and the control delay do not both enter the diffusion term of the state equation. This paper aims to propose a general model, which can cover all the above cases, and gives the state feedback representation of an open-loop Stackelberg strategy to such a differential game. In summary, the contributions and innovations of this paper are as follows.
\begin{itemize}
 \item Our model is general. The state equations of leader and follower both contain state delay and control delay, moreover, which both enter into the diffusion terms.
 \item Our conclusions are simple. The open-loop stackelberg strategies are derived via two Pseudo-Riccati equations (see (\ref{P1-equation}), (\ref{P2-equation})) and a matrix equation (\ref{L-equation}), which are easier to solve than the Riccati equations (6), (19) and the matrix equation (25) in \cite{XSZ18}.
 \item Our theory is efficient. Two examples are addressed applying the above theoretical results.
\end{itemize}

The paper is organized as follows. Section 2 formulates the linear quadratic stochastic Stackelberg differential game with time delay. The optimization problem of the follower is discussed in Section 3, and then the optimization problem of the leader is addressed in Section 4. To illustrate the applications of the theoretical results, two examples are considered in Section 5. Finally Section 6 gives some concluding remarks.

\section{Problem statement}

We first introduce some notations which will be used in this paper and then state the problem.

Throughout the paper, $\mathbf{R}^{n\times m}$ is the Euclidean space of all $n\times m$ real matrices, $\mathbf{S}^n$ is the space of all $n\times n$ symmetric matrices. We simply write $\mathbf{R}^{n\times m}$ as $\mathbf{R}^n$ when $m=1$. The norm in $\mathbf{R}^n$ is denoted by $|\cdot|$ and the inner product is denoted by $\langle\cdot,\cdot\rangle$. The transpose of vectors or matrices is denoted by the superscript $^\top$. $I$ is the identity matrix with appropriate dimension.

Let $T>0$ be fixed and $\delta>0$ be the time delay, suppose $(\Omega,\mathcal{F},\mathbb{P})$ is a complete probability space, $\{W(t)\}_{t\geq 0}$ is a one-dimensional Brownian motion, and we define the filtration $\{\mathcal{F}_t\}_{t\geq 0}=\sigma\{W(r);0 \leq r\leq t\}$. $\mathbb{E}[\cdot]$ denotes the expectation under the probability measure $\mathbb{P}$ and $\mathbb{E}^{\mathcal{F}_t}[\cdot]:=\mathbb{E}[\cdot|\mathcal{F}_t]$ denotes the conditional expectation, for any $t\in(0,T]$.

\vspace{1mm}

We then introduce some spaces. For a positive integer $p$, we define
\begin{eqnarray*}\begin{aligned}
   &L^p([0,T];\mathbf{R}^{n\times n}):=\bigg\{\mathbf{R}^{n\times n}\mbox{-valued funciton }\phi(t);\ \int_0^T|\phi(t)|^pdt<\infty\bigg\},\\
   &L^{\infty}([0,T];\mathbf{R}^{n\times n}):=\bigg\{\mathbf{R}^{n\times n}\mbox{-valued funciton }\phi(t);\ \sup\limits_{0\leq t\leq T}|\phi(t)|dt<\infty\bigg\},\\
   &C([0,T];\mathbf{R}^n):=\bigg\{\mathbf{R}^n\mbox{-valued continuous funciton }\phi(t);\ \sup\limits_{0\leq t\leq T}|\phi(t)|<\infty\bigg\},\\
   &L^2_\mathcal{F}([0,T];\mathbf{R}^n):=\bigg\{\mathbf{R}^n\mbox{-valued }\mathcal{F}_t\mbox{-adapted process }\phi(t);\ \mathbb{E}\int_0^T|\phi(t)|^2dt<\infty\bigg\},\\
   &L^2_\mathcal{F}(\Omega;C([0,T];\mathbf{R}^n)):=\bigg\{\mathbf{R}^n\mbox{-valued }\mathcal{F}_t\mbox{-adapted process }\phi(t);\
    \mathbb{E}\Big[\sup\limits_{0\leq t\leq T}|\phi(t)|^2\Big]<\infty\bigg\}.
\end{aligned}\end{eqnarray*}

\vspace{1mm}

Consider the following linear controlled system with time delay, which is a {\it stochastic differential delayed equation} (in short, SDDE):
\small
\begin{equation}\label{system equation}\left\{\begin{aligned}
 dX^{u_1,u_2}(t)&=\big[A(t)X^{u_1,u_2}(t)+\bar{A}(t)X^{u_1,u_2}(t-\delta)+\bar{B}_1(t)u_1(t-\delta)+\bar{B}_2(t)u_2(t-\delta)\big]dt\\
                &\quad+\big[C(t)X^{u_1,u_2}(t)+\bar{C}(t)X^{u_1,u_2}(t-\delta)+\bar{D}_1(t)u_1(t-\delta)+\bar{D}_2(t)u_2(t-\delta)\big]dW(t),\\
                &\hspace{10cm} t\in[0,T],\\
  X^{u_1,u_2}(t)&=\varphi(t),\ u_1(t)=\eta_1(t),\ u_2(t)=\eta_2(t),\quad t\in[-\delta,0],
\end{aligned}\right.\end{equation}
\normalsize
and the cost functionals for the leader and the follower are as follows, respectively:
\small
\begin{equation}\begin{aligned}\label{cost}
  &J_i(u_1(\cdot),u_2(\cdot))=\mathbb{E}\bigg[\int_0^T\big[X^{u_1,u_2}(t)^\top Q_i(t)X^{u_1,u_2}(t)+X^{u_1,u_2}(t-\delta)^\top\bar{Q}_i(t)X^{u_1,u_2}(t-\delta)\\
  &\quad+u_i(t)^\top R_i(t)u_i(t)+u_i(t-\delta)^\top\bar{R}_i(t)u_i(t-\delta)\big]dt+X^{u_1,u_2}(T)^\top G_iX^{u_1,u_2}(T)\bigg],\quad i=1,2.
\end{aligned}\end{equation}
\normalsize
In the above, $X(\cdot)\in\textbf{R}^n$ is the state process, $u_1(\cdot)\in\textbf{R}^{k_1}$ and $u_2(\cdot)\in\textbf{R}^{k_2}$ are the control processes of the follower and the leader, respectively. $\varphi(\cdot)\in C([-\delta,0];\textbf{R}^n)$ is the initial trajectory of the state, $\eta_i(\cdot)\in L^2([-\delta,0];\textbf{R}^{k_i})$, $i=1,2$, are the initial trajectories of the follower's and the leader's control, respectively. $A(\cdot),\bar{A}(\cdot),C(\cdot),\bar{C}(\cdot)\in L^{\infty}([0,T];\textbf{R}^{n\times n})$, $\bar{B}_i(\cdot),\bar{D}_i(\cdot)\in L^{\infty}([0,T];\textbf{R}^{n\times k_i})$, $Q_i(\cdot),\bar{Q}_i(\cdot)\in L^{\infty}([0,T];\textbf{S}^n)$, $R_i(\cdot),\bar{R}_i(\cdot)\in L^{\infty}([0,T];\textbf{S}^{k_i})$, and $G_i\in \textbf{S}^n$.

For any $u_1(\cdot)\in L_{\mathcal{F}}^2([0,T];\textbf{R}^{k_1})$, $u_2(\cdot)\in L_{\mathcal{F}}^2([0,T];\textbf{R}^{k_2})$, (\ref{system equation}) admits a unique solution $X^{u_1,u_2}(\cdot)\in L_{\mathcal{F}}^2(\Omega;C([0,T];\textbf{R}^n))$ (see \cite{Mohammed84}) and the cost functionals are well-defined.

Our {\it linear quadratic stochastic Stackelberg differential game with time delay} is the following.
For each choice of the leader $u_2(\cdot)\in\mathcal{U}_2[0,T]\triangleq L_{\mathcal{F}}^2([0,T];\textbf{R}^{k_2})$, the follower would like to choose a strategy $\bar{u}_1(\cdot)\in\mathcal{U}_1[0,T]\triangleq L_{\mathcal{F}}^2([0,T];\textbf{R}^{k_1})$ such that his cost functional $J_1(\bar{u}_1(\cdot),u_2(\cdot))$ is the minimum of $J_1(u_1(\cdot),u_2(\cdot))$ over $u_1(\cdot)\in\mathcal{U}_1[0,T]$. The leader has the ability to know the optimal strategy of the follower $\bar{u}_1(\cdot)$, thus the leader would like to choose a strategy $\bar{u}_2(\cdot)\in\mathcal{U}_2[0,T]$ such that his cost functional $J_2(\bar{u}_1(\cdot),\bar{u}_2(\cdot))$ is the minimum of $J_2(\bar{u}_1(\cdot),u_2(\cdot))$ over $u_2(\cdot)\in\mathcal{U}_2[0,T]$. Strictly speaking, the follower wants to find a map $\bar{\alpha}_1[\cdot]:\mathcal{U}_2[0,T]\rightarrow\mathcal{U}_1[0,T]$ and the leader wants to find a control $\bar{u}_2(\cdot)\in\mathcal{U}_2[0,T]$, such that
\begin{eqnarray*}\left\{\begin{aligned}
  &J_1(\bar{\alpha}_1(u_2(\cdot))(\cdot),u_2(\cdot))=\underset{u_1(\cdot)\in\mathcal{U}_1[0,T]}{\inf}J_1(u_1(\cdot),u_2(\cdot)),\quad \forall u_2(\cdot)\in\mathcal{U}_2[0,T],\\
  &J_2(\bar{\alpha}_1(\bar{u}_2(\cdot))(\cdot),\bar{u}_2(\cdot))=\underset{u_2(\cdot)\in\mathcal{U}_2[0,T]}{\inf}J_2(\bar{\alpha}_1(u_2(\cdot)),u_2(\cdot)).
\end{aligned}\right.\end{eqnarray*}
Then $\bar{u}_1(\cdot)\triangleq\bar{\alpha}_1[\bar{u}_2(\cdot)]$, the pair $(\bar{u}_1(\cdot),\bar{u}_2(\cdot))$ is called an {\it open-loop Stackelberg strategy} to the above game, and the corresponding solution $\bar{X}(\cdot)\equiv X^{\bar{u}_1,\bar{u}_2}(\cdot)$ is called the {\it optimal state trajectory}. The purpose of this paper is to find and characterize the unique open-loop Stackelberg strategy.

\section{Optimization problem of the follower}

In this section, we deal with the optimization problem of the follower, which is a linear quadratic stochastic optimal control problem with time delay, for any choice $u_2(\cdot)$ of the leader.

We give the following detailed statement.

Problem \textbf{(F-DLQ)}: {\it For any $u_2(\cdot)\in\mathcal{U}_2[0,T]$, minimize the cost functional $J_1(u_1(\cdot),u_2(\cdot))$ over $u_1(\cdot)\in\mathcal{U}_1[0,T]$ such that (\ref{system equation}) is satisfied.}

\vspace{1mm}

To solve Problem \textbf{(F-DLQ)}, we first introduce the adjoint equation:
\begin{equation}\left\{\begin{aligned}\label{p-equation}
  dp(t)&=\Big\{-A(t)^\top p(t)-C(t)^\top q(t)-\big[Q_1(t)+\bar{Q}_1(t+\delta)\big]X^{\bar{u}_1,u_2}(t)\\
       &\qquad-\mathbb{E}^{\mathcal{F}_t}\big[\bar{A}(t+\delta)^\top p(t+\delta)+\bar{C}(t+\delta)^\top q(t+\delta)\big]\Big\}dt+q(t)dW(t),\ t\in[0,T],\\
  p(T)&=G_1X^{\bar{u}_1,u_2}(T),\ p(t)=q(t)=0,\ t\in(T,T+\delta].
\end{aligned}\right.\end{equation}
(\ref{p-equation}) is called an \emph{anticipated backward stochastic differential equation} (in short, ABSDE), which has a natural adjoint relation with the SDDE (\ref{system equation}) (see \cite{CW10}). By Peng and Yang \cite{PY09}, (\ref{p-equation}) admits a unique solution $(p(\cdot),q(\cdot))\in L_{\mathcal{F}}^2(\Omega;C([0,T];\textbf{R}^n))\times L_{\mathcal{F}}^2([0,T];\textbf{R}^n)$.

\vspace{1mm}

Then we can obtain the following result.
\begin{mythm}\label{thm3.1}
Assume $\bar{Q}_1(t)=\bar{R}_1(t)=0$ for $t\in [T,T+\delta]$. For any $u_2(\cdot)\in\mathcal{U}_2[0,T]$, let $\bar{u}_1(\cdot)$ be an optimal control of the follower and $X^{\bar{u}_1,u_2}(\cdot)$ be the optimal state trajectory for Problem (\textbf{F-DLQ}), then
\begin{equation}\label{u1-optimal}\begin{aligned}
  &\big[R_1(t)+\bar{R}_1(t+\delta)\big]\bar{u}_1(t)+\mathbb{E}^{\mathcal{F}_t}\big[\bar{B}_1(t+\delta)^Tp(t+\delta)+\bar{D}_1(t+\delta)^Tq(t+\delta)\big]=0,\\
  &\hspace{9cm} a.e.\ t\in[0,T],\ \mathbb{P}\mbox{-}a.s.
\end{aligned}\end{equation}
\end{mythm}
\begin{proof}
For any $v_1(\cdot)\in\mathcal{U}_1[0,T]$ with $v_1(t)=\eta_1(t)$ for $t\in[-\delta,0]$, define $u_1^\varepsilon(\cdot)=\bar{u}_1(\cdot)+\varepsilon(v_1(\cdot)-\bar{u}_1(\cdot))$, $\varepsilon\in[0,1]$. Suppose $X^{v_1,u_2}(\cdot)$ and $X^{u_1^\varepsilon,u_2}(\cdot)$ are the state trajectories corresponding to $v_1(\cdot)$ and $u_1^\varepsilon(\cdot)$, respectively. Then we have
\begin{equation}\begin{aligned}\label{eq3.3}
  &J_1(u_1^\varepsilon(\cdot),u_2(\cdot))-J_1(\bar{u}_1(\cdot),u_2(\cdot))\\
  &=2\varepsilon\mathbb{E}\bigg\{\int_0^T\Big[\big\langle Q_1(t)X^{\bar{u}_1,u_2}(t),X^{v_1,u_2}(t)-X^{\bar{u}_1,u_2}(t)\big\rangle\\
  &\qquad\qquad\qquad+\big\langle\bar{Q}_1(t)X^{\bar{u}_1,u_2}(t-\delta),X^{v_1,u_2}(t-\delta)-X^{\bar{u}_1,u_2}(t-\delta)\big\rangle\\
  &\qquad\qquad\qquad+\big\langle R_1(t)\bar{u}_1(t),v_1(t)-\bar{u}_1(t)\big\rangle+\big\langle\bar{R}_1(t)\bar{u}_1(t-\delta),v_1(t-\delta)-\bar{u}_1(t-\delta)\big\rangle\Big]dt\\
  &\qquad\qquad+\big\langle G_1X^{\bar{u}_1,u_2}(T),X^{v_1,u_2}(T)-X^{\bar{u}_1,u_2}(T)\big\rangle\bigg\}+\varepsilon^2\Big\{\cdots\Big\}.
\end{aligned}\end{equation}
Applying It\^o's formula to $\big\langle p(\cdot),X^{v_1,u_2}(\cdot)-X^{\bar{u}_1,u_2}(\cdot)\big\rangle$, and substituting it into (\ref{eq3.3}), we deduce
\begin{equation}\begin{aligned}\label{eq3.4}
  &J_1(u_1^\varepsilon(\cdot),u_2(\cdot))-J_1(\bar{u}_1(\cdot),u_2(\cdot))\\
  &=2\varepsilon\mathbb{E}\int_0^T\Big[\big\langle\big[R_1(t)+\bar{R}_1(t+\delta)\big]\bar{u}_1(t)+\mathbb{E}^{\mathcal{F}_t}\big[\bar{B}_1(t+\delta)^Tp(t+\delta)\\
  &\qquad\qquad\quad+\bar{D}_1(t+\delta)^Tq(t+\delta)\big],v_1(t)-\bar{u}_1(t)\big\rangle\Big]dt+\varepsilon^2\Big\{\cdots\Big\}.
\end{aligned}\end{equation}
Dividing both sides of (\ref{eq3.4}) by $\varepsilon$, we derive
\begin{equation*}\begin{aligned}
  0&\geq\frac{1}{\varepsilon}\big[J_1(u_1^\varepsilon(\cdot),u_2(\cdot))-J_1(\bar{u}_1(\cdot),u_2(\cdot))\big]\\
  &=2\mathbb{E}\int_0^T\Big[\big\langle\big[R_1(t)+\bar{R}_1(t+\delta)\big]\bar{u}_1(t)+\mathbb{E}^{\mathcal{F}_t}\big[\bar{B}_1(t+\delta)^Tp(t+\delta)\\
  &\qquad\qquad\quad+\bar{D}_1(t+\delta)^Tq(t+\delta)\big],v_1(t)-\bar{u}_1(t)\big\rangle\Big]dt+\varepsilon\Big\{\cdots\Big\}.
\end{aligned}\end{equation*}
Finally letting $\varepsilon\rightarrow 0$, then due to the arbitrariness chosen of $v_1(\cdot)$, we complete the proof.
\end{proof}

\vspace{1mm}

Next we try to give the feedback expression of $\bar{u}_1(\cdot)$. For this target, for some technical reason, we have to impose the following assumptions on the coefficients of (\ref{system equation}) and (\ref{cost}):
\begin{equation*}
  \textbf{(A1)}\left\{\begin{aligned}
  &C(t)^\top P_1(t)\bar{D}_1(t)+P_1(t)\bar{B}_1(t)=0,\quad \ t\in[0,T],\\
  &C(t)^\top P_1(t)\bar{C}(t)+P_1(t)\bar{A}(t)=0,\quad \ t\in[0,T],
  \end{aligned}\right.
\end{equation*}
where $P_1(\cdot)$ is the solution to the following Pseudo-Riccati equation:
\begin{equation}\left\{\begin{aligned}\label{P1-equation}
  \dot{P}_1(t)&=-P_1(t)A(t)-A(t)^\top P_1(t)-C(t)^\top P_1(t)C(t)\\
              &\quad-\bar{C}(t+\delta)^\top P_1(t+\delta)\bar{C}(t+\delta)-Q_1(t)-\bar{Q}_1(t+\delta)\\
              &\quad+\bar{C}(t+\delta)^\top P_1(t+\delta)\bar{D}_1(t+\delta)\Omega_1^{-1}(t)\bar{D}_1(t+\delta)^\top P_1(t+\delta)\bar{C}(t+\delta),\ t\in[0,T],\\
        P_1(T)&=G_1,\ P_1(t)=0,\quad t\in(T,T+\delta],\\
   \Omega_1(t)&\triangleq R_1(t)+\bar{R}_1(t+\delta)+\bar{D}_1(t+\delta)^\top P_1(t+\delta)\bar{D}_1(t+\delta)>0,\ \forall\ t\in[0,T].
\end{aligned}\right.\end{equation}
And we choose $(\zeta_1(\cdot),\bar{\zeta}_1(\cdot))$ satisfying the following ABSDE:
\begin{equation}\left\{\begin{aligned}\label{zeta1-equation}
  d\zeta_1(t)&=\Big\{-A(t)^\top\zeta_1(t)-C(t)^\top\bar{\zeta}_1(t)-\mathbb{E}^{\mathcal{F}_t}\big[\bar{A}(t+\delta)^\top\zeta_1(t+\delta)+\bar{C}(t+\delta)^\top\bar{\zeta}_1(t+\delta)\big]\\
             &\qquad+\bar{C}(t+\delta)^\top P_1(t+\delta)\bar{D}_1(t+\delta)\Omega_1^{-1}(t)\big\{\mathbb{E}^{\mathcal{F}_t}\big[\bar{B}_1(t+\delta)^\top\zeta_1(t+\delta)\\
             &\qquad\quad+\bar{D}_1(t+\delta)^\top\bar{\zeta}_1(t+\delta)\big]+\bar{D}_1(t+\delta)^\top P_1(t+\delta)\bar{D}_2(t+\delta)u_2(t)\big\}\\
             &\qquad-\big[P_1(t)\bar{B}_2(t)+C(t)^\top P_1(t)\bar{D}_2(t)\big]u_2(t-\delta)\\
             &\qquad-\bar{C}(t+\delta)^\top P_1(t+\delta)\bar{D}_2(t+\delta)u_2(t)\Big\}dt+\bar{\zeta}_1(t)dW(t),\ t\in[0,T],\\
   \zeta_1(t)&=\bar{\zeta}_1(t)=0,\ t\in(T,T+\delta].
\end{aligned}\right.\end{equation}

\begin{Remark}\label{rem3.1}
We point out that (\ref{P1-equation}) is not a Riccati-type equation. In fact, (\ref{P1-equation}) can be solved step by step as a linear {\it ordinary differential equation} (in short, ODE). Furthermore, its unique solvability can be guaranteed by the boundedness of its coefficients. In details, for $t\in(T-\delta,T]$, (\ref{P1-equation}) becomes
\begin{equation*}\left\{\begin{aligned}
  \dot{P}_1(t)&=-P_1(t)A(t)-A(t)^\top P_1(t)-C(t)^\top P_1(t)C(t)-Q_1(t)-\bar{Q}_1(t+\delta),\ t\in(T-\delta,T],\\
        P_1(T)&=G_1.
\end{aligned}\right.\end{equation*}
Obviously, it admits a unique solution $P_1(\cdot)\in C((T-\delta,T];\textbf{S}^n)$. Repeating the above steps on $(T-2\delta,T-\delta]$, $(T-3\delta,T-2\delta]$, $\cdots$, we can derive the solution to (\ref{P1-equation}) on $[0,T]$. On the other hand, due to the boundedness of the coefficients, ABSDE (\ref{zeta1-equation}) has a unique solution $(\zeta_1(\cdot),\bar{\zeta}_1(\cdot))\in L_{\mathcal{F}}^2(\Omega;C([0,T];\textbf{R}^n))\times L_{\mathcal{F}}^2([0,T];\textbf{R}^n)$ for any $u_2(\cdot)\in\mathcal{U}_2[0,T]$ (see \cite{PY09}).
\end{Remark}

\begin{Remark}\label{rem3.1-1}
From Remark \ref{rem3.1}, we know (\ref{P1-equation}) is essentially a linear ODE and hence is easier to solve than the Riccati equation (6) in \cite{XSZ18}.
\end{Remark}

\vspace{1mm}

Next we can derive the sufficient and necessary conditions of the solvability for Problem \textbf{(F-DLQ)}.

\begin{mythm}\label{thm3.2}
Let \textbf{(A1)} hold, and $\bar{Q}_1(t)=\bar{R}_1(t)=0$ for $t\in [T,T+\delta]$. Let $P_1(\cdot)$ satisfy (\ref{P1-equation}) and $(\zeta_1(\cdot),\bar{\zeta}_1(\cdot))$ satisfy (\ref{zeta1-equation}). Then for any $u_2(\cdot)\in\mathcal{U}_2[0,T]$, Problem \textbf{(F-DLQ)} is solvable
and the optimal control $\bar{u}_1(\cdot)$ is of the following state feedback form:
\begin{equation}\begin{aligned}\label{u1-feedback}
  \bar{u}_1(t)=&-\Omega_1^{-1}(t)\Big\{\bar{D}_1(t+\delta)^\top P_1(t+\delta)\bar{C}(t+\delta)X^{\bar{u}_1,u_2}(t)\\
               &\qquad\qquad+\mathbb{E}^{\mathcal{F}_t}\big[\bar{B}_1(t+\delta)^\top\zeta_1(t+\delta)+\bar{D}_1(t+\delta)^\top\bar{\zeta}_1(t+\delta)\big]\\
               &\qquad\qquad+\bar{D}_1(t+\delta)^\top P_1(t+\delta)\bar{D}_2(t+\delta)u_2(t)\Big\},\ a.e.\  t\in[0,T],\ \mathbb{P}\mbox{-}a.s.
\end{aligned}\end{equation}
Moreover, the optimal cost can be obtained as follows:
\begin{equation}\begin{aligned}\label{follower-optimal-cost}
  &\underset{u_1(\cdot)\in\mathcal{U}_1[0,T]}{\inf}J_1(u_1(\cdot),u_2(\cdot))=J_1(\bar{u}_1(\cdot),u_2(\cdot))\\
  &=\mathbb{E}\bigg\{\big\langle P_1(0)\varphi(0)+2\zeta_1(0),\varphi(0)\big\rangle+\int_{-\delta}^0I_\delta(t)dt+\int_0^T\Big[\big\langle P_1(t)\bar{D}_2(t)u_2(t-\delta),\bar{D}_2(t)u_2(t-\delta)\big\rangle\\
  &\qquad+2\langle\bar{B}_2(t)^\top\zeta_1(t)+\bar{D}_2(t)^\top\bar{\zeta}_1(t),u_2(t-\delta)\big\rangle-\big|\Omega_1^{-\frac{1}{2}}(t)\mathbb{E}^{\mathcal{F}_t}\big[\bar{B}_1(t+\delta)^\top\zeta_1(t+\delta)\\
  &\qquad+\bar{D}_1(t+\delta)^\top\bar{\zeta}_1(t+\delta)
   +\bar{D}_1(t+\delta)^\top P_1(t+\delta)\bar{D}_2(t+\delta)u_2(t)\big]\big|^2\Big]dt\bigg\},
\end{aligned}\end{equation}
where
\begin{equation}\begin{aligned}\label{I-delta}
  I_\delta(t)&\triangleq\big\langle\varphi(t),\bar{C}(t+\delta)^\top P_1(t+\delta)\bar{C}(t+\delta)\varphi(t)+\bar{Q}_1(t+\delta)\varphi(t)+\bar{C}(t+\delta)^\top\bar{\zeta}_1(t+\delta)\\
             &\qquad+\bar{A}(t+\delta)^\top\zeta_1(t+\delta)\big\rangle+\big\langle\eta_1(t),2\bar{B}_1(t+\delta)^\top\zeta_1(t+\delta)+\bar{D}_1(t+\delta)^\top\bar{\zeta}_1(t+\delta)\\
             &\qquad+\bar{R}_1(t+\delta)\eta_1(t)+\bar{D}_1(t+\delta)^\top P_1(t+\delta)\bar{D}_1(t+\delta)\eta_1(t)\big\rangle\\
             &\quad+2\big\langle\eta_1(t),\bar{D}_1(t)^\top P_1(t)\bar{C}(t)\varphi(t)\big\rangle+2\big\langle\eta_1(t),\bar{D}_1(t+\delta)^\top P_1(t+\delta)\bar{D}_2(t+\delta)\eta_2(t)\big\rangle\\
             &\quad+2\big\langle\eta_2(t),\bar{D}_2(t+\delta)^\top P_1(t+\delta)\bar{C}(t+\delta)\varphi(t)\big\rangle.
\end{aligned}\end{equation}
\end{mythm}
\begin{proof}
For any $u_1(\cdot)\in\mathcal{U}_1[0,T]$, applying It\^o's formula to $\big\langle P_1(\cdot)X^{u_1,u_2}(\cdot),X^{u_1,u_2}(\cdot)\big\rangle$ and $2\big\langle\zeta_1(\cdot),\\X^{u_1,u_2}(\cdot)\big\rangle$, with some computation, by \textbf{(A1)} we deduce
\begin{equation*}\begin{aligned}
  &J_1(u_1(\cdot),u_2(\cdot))\\
  &=\mathbb{E}\bigg\{\big\langle P_1(0)\varphi(0)+2\zeta_1(0),\varphi(0)\big\rangle+\int_{-\delta}^0I_\delta(t)dt+\int_0^T\Big[\big\langle P_1(t)\bar{D}_2(t)u_2(t-\delta),\bar{D}_2(t)u_2(t-\delta)\big\rangle\\
  &\qquad+2\big\langle\bar{B}_2(t)^\top\zeta_1(t)+\bar{D}_2(t)^\top\bar{\zeta}_1(t),u_2(t-\delta)\big\rangle-\big|\Omega_1^{-\frac{1}{2}}(t)\mathbb{E}^{\mathcal{F}_t}\big[\bar{B}_1(t+\delta)^\top\zeta_1(t+\delta)\\
  &\qquad\quad+\bar{D}_1(t+\delta)^\top\bar{\zeta}_1(t+\delta)+\bar{D}_1(t+\delta)^\top P_1(t+\delta)\bar{D}_2(t+\delta)u_2(t)\big]\big|^2\\
  &\qquad+\Omega_1(t)\big|u_1(t)+\Omega_1^{-1}(t)\big\{\bar{D}_1(t+\delta)^\top P_1(t+\delta)\bar{C}(t+\delta)X^{u_1,u_2}(t)+\mathbb{E}^{\mathcal{F}_t}\big[\bar{B}_1(t+\delta)^\top\zeta_1(t+\delta)\\
  &\qquad\quad+\bar{D}_1(t+\delta)^\top\bar{\zeta}_1(t+\delta)\big]+\bar{D}_1(t+\delta)^\top P_1(t+\delta)\bar{D}_2(t+\delta)u_2(t)\big\}\big|^2\Big]dt\bigg\},
\end{aligned}\end{equation*}
which implies (\ref{u1-feedback}) is the optimal control and (\ref{follower-optimal-cost}) holds.
\end{proof}

\begin{Remark}\label{rem3.2}
In order to overcome the difficulty caused by the time delay, in this paper we impose the assumption \textbf{(A1)}. Otherwise, we will get the feedback in the form of
\begin{equation*}
  \bar{u}_1(t)=\{\cdots\}X^{\bar{u}_1,u_2}(t)+\{\cdots\}\mathbb{E}^{\mathcal{F}_t}[X^{\bar{u}_1,u_2}(t+\delta)],
\end{equation*}
which is more complex than (\ref{u1-feedback}) (in which the terms like $\mathbb{E}^{\mathcal{F}_t}[X^{\bar{u}_1,u_2}(t+\delta)]$ disappears), and leads to our inability to deal with the optimization problem for the leader. In the future we will consider some weaker assumptions than \textbf{(A1)} to deal with this problem.
\end{Remark}

\section{Optimization problem of the leader}

In this section, we will address the optimization problem of the leader, which is a linear quadratic stochastic optimal control problem with a state equation formed by an SDDE and an ABSDE.

\vspace{1mm}

Now the follower takes the optimal control $\bar{u}_1(\cdot)$ of form $(\ref{u1-feedback})$, consequently the leader has the following state equation:
\begin{equation}\label{leader state}\left\{\begin{aligned}
 dX^{\bar{u}_1,u_2}(t)&=\Big\{\hat{A}_1(t)X^{\bar{u}_1,u_2}(t)+\hat{A}_2(t)X^{\bar{u}_1,u_2}(t-\delta)+\hat{B}(t)u_2(t-\delta)+\hat{F}(t)\mathbb{E}^{\mathcal{F}_{t-\delta}}[\zeta_1(t)]\\
                      &\qquad+\hat{H}(t)\mathbb{E}^{\mathcal{F}_{t-\delta}}[\bar{\zeta}_1(t)]+\bar{B}_1(t)\eta_1(t-\delta)I_{[0,\delta]}(t)\Big\}dt\\
                      &\quad+\Big\{\hat{C}_1(t)X^{\bar{u}_1,u_2}(t)+\hat{C}_2(t)X^{\bar{u}_1,u_2}(t-\delta)+\hat{D}(t)u_2(t-\delta)+\hat{K}(t)\mathbb{E}^{\mathcal{F}_{t-\delta}}[\zeta_1(t)]\\
                      &\qquad+\hat{M}(t)\mathbb{E}^{\mathcal{F}_{t-\delta}}[\bar{\zeta}_1(t)]+\bar{D}_1(t)\eta_1(t-\delta)I_{[0,\delta]}(t)\Big\}dW(t),\\
           d\zeta_1(t)&=\Big\{-\hat{A}_1(t)^\top\zeta_1(t)-\hat{C}_1(t)^\top\bar{\zeta}_1(t)-\mathbb{E}^{\mathcal{F}_t}[\tilde{A}_2(t+\delta)^\top\zeta_1(t+\delta)+\tilde{C}_2(t+\delta)^\top\bar{\zeta}_1(t+\delta)]\\
                      &\qquad+\hat{N}_1(t)u_2(t)+\hat{N}_2(t)u_2(t-\delta)\Big\}dt+\bar{\zeta}_1(t)dW(t),\ t\in[0,T],\\
  X^{\bar{u}_1,u_2}(t)&=\varphi(t),\ u_2(t)=\eta_2(t),\ t\in[-\delta,0],\\
            \zeta_1(t)&=\bar{\zeta}_1(t)=0,\ t\in[T,T+\delta],
\end{aligned}\right.\end{equation}
where
\begin{equation*}\left\{\begin{aligned}
  &\hat{A}_1(t)\triangleq A(t),\qquad \hat{C}_1(t)\triangleq C(t),\qquad \hat{A}_2(t)\triangleq\bar{A}(t)-\bar{B}_1(t)I_{[\delta,T]}(t)\Omega_1^{-1}(t-\delta)\bar{D}_1(t)^\top P_1(t)\bar{C}(t),\\
  &\tilde{A}_2(t)\triangleq\bar{A}(t)-\bar{B}_1(t)\Omega_1^{-1}(t-\delta)\bar{D}_1(t)^\top P_1(t)\bar{C}(t),\\
  &\hat{B}(t)\triangleq\bar{B}_2(t)-\bar{B}_1(t)I_{[\delta,T]}(t)\Omega_1^{-1}(t-\delta)\bar{D}_1(t)^\top P_1(t)\bar{D}_2(t),\\
  &\hat{C}_2(t)\triangleq\bar{C}(t)-\bar{D}_1(t)I_{[\delta,T]}(t)\Omega_1^{-1}(t-\delta)\bar{D}_1(t)^\top P_1(t)\bar{C}(t),\\
  &\tilde{C}_2(t)\triangleq\bar{C}(t)-\bar{D}_1(t)\Omega_1^{-1}(t-\delta)\bar{D}_1(t)^\top P_1(t)\bar{C}(t),\\
  &\hat{D}(t)\triangleq\bar{D}_2(t)-\bar{D}_1(t)I_{[\delta,T]}(t)\Omega_1^{-1}(t-\delta)\bar{D}_1(t)^\top P_1(t)\bar{D}_2(t),\\
  &\hat{F}(t)\triangleq-\bar{B}_1(t)I_{[\delta,T]}(t)\Omega_1^{-1}(t-\delta)\bar{B}_1(t)^\top,\quad\quad\hat{H}(t)\triangleq-\bar{B}_1(t)I_{[\delta,T]}(t)\Omega_1^{-1}(t-\delta)\bar{D}_1(t)^\top,\\
  &\hat{K}(t)\triangleq-\bar{D}_1(t)I_{[\delta,T]}(t)\Omega_1^{-1}(t-\delta)\bar{B}_1(t)^\top,\quad\quad\hat{M}(t)\triangleq-\bar{D}_1(t)I_{[\delta,T]}(t)\Omega_1^{-1}(t-\delta)\bar{D}_1(t)^\top,\\
  &\hat{N}_1(t)\triangleq-\bar{C}(t+\delta)^\top P_1(t+\delta)\big[I-\bar{D}_1(t+\delta)\Omega_1^{-1}(t)\bar{D}_1(t+\delta)^\top P_1(t+\delta)\big]\bar{D}_2(t+\delta),\\
  &\hat{N}_2(t)\triangleq-P_1(t)\bar{B}_2(t)-C(t)^\top P_1(t)\bar{D}_2(t).
\end{aligned}\right.\end{equation*}

\vspace{1mm}

Now we formulate the optimization problem of the leader as follows.

Problem \textbf{(L-DLQ)}: {\it Minimize the cost functional $J_2(\bar{u}_1(\cdot),u_2(\cdot))$ over $u_2(\cdot)\in\mathcal{U}_2[0,T]$ such that (\ref{leader state}) is satisfied.}

\vspace{1mm}

First we introduce the adjoint equation as follows:
\begin{equation}\left\{\begin{aligned}\label{p2-equation}
  d\xi(t)&=\Big\{\hat{A}_1(t)\xi(t)+\tilde{A}_2(t)\xi(t-\delta)+\mathbb{E}^{\mathcal{F}_{t-\delta}}\big[\hat{F}(t)^\top p_2(t)+\hat{K}(t)^\top q_2(t)\big]\Big\}dt\\
         &\quad+\Big\{\hat{C}_1(t)\xi(t)+\tilde{C}_2(t)\xi(t-\delta)+\mathbb{E}^{\mathcal{F}_{t-\delta}}\big[\hat{H}(t)^\top p_2(t)+\hat{M}(t)^\top q_2(t)\big]\Big\}dW(t),\\
  dp_2(t)&=\Big\{-\hat{A}_1(t)^\top p_2(t)-\hat{C}_1(t)^\top q_2(t)-\big[Q_2(t)+\bar{Q}_2(t+\delta)\big]\bar{X}(t)\\
         &\qquad-\mathbb{E}^{\mathcal{F}_t}\big[\hat{A}_2(t+\delta)^\top p_2(t+\delta)+\hat{C}_2(t+\delta)^\top q_2(t+\delta)\big]\Big\}dt+q_2(t)dW(t),\ t\in[0,T],\\
   \xi(t)&=0,\ t\in[-\delta,0],\\
   p_2(T)&=G_2\bar{X}(T),\ p_2(t)=0,\ t\in(T,T+\delta],\ q_2(t)=0,\ t\in[T,T+\delta],
\end{aligned}\right.\end{equation}
where $\bar{X}(\cdot)\equiv X^{\bar{u}_1,\bar{u}_2}(\cdot)$. The adjoint equation (\ref{p2-equation}) is formed by an SDDE and an ABSDE, which is partial coupled. By \cite{PY09} the second equation in (\ref{p2-equation}) admits a unique solution $(p_2(\cdot),q_2(\cdot))\in L_{\mathcal{F}}^2(\Omega;C([0,T];\textbf{R}^n))\times L_{\mathcal{F}}^2([0,T];\textbf{R}^n)$, then by \cite{Mohammed84} the first equation in (\ref{p2-equation}) admits a unique solution $\xi(\cdot)\in L_{\mathcal{F}}^2(\Omega;C([0,T];\textbf{R}^n))$.

\vspace{1mm}

Then we can obtain the following result similar to Theorem \ref{thm3.1}.
\begin{mythm}\label{thm4.1}
Assume $\bar{Q}_2(t)=\bar{R}_2(t)=0$ for $t\in [T,T+\delta]$. Let $\bar{u}_2(\cdot)$ be the optimal control of the leader and $\bar{X}(\cdot)$ be the optimal state strategy for the Problem (\textbf{L-DLQ}), then $\bar{u}_2(\cdot)$ satisfies
\begin{equation}\begin{aligned}\label{u2-optimal}
  &\big[R_2(t)+\bar{R}_2(t+\delta)\big]\bar{u}_2(t)+\mathbb{E}^{\mathcal{F}_t}\big[\hat{B}(t+\delta)^\top p_2(t+\delta)+\hat{D}(t+\delta)^\top q_2(t+\delta)\\
  &-\hat{N}_2(t+\delta)^\top\xi(t+\delta)\big]-\hat{N}_1(t)^\top\xi(t)=0,\quad a.e.\ t\in[0,T],\ \mathbb{P}\mbox{-}a.s.
\end{aligned}\end{equation}
\end{mythm}

\vspace{1mm}

Next we would like to give the feedback expression of $\bar{u}_2(\cdot)$. In addition \textbf{(A1)}, now we have to impose some other assumptions to the coefficients of (\ref{system equation}) and (\ref{cost}):
\begin{equation*}
  \textbf{(A2)}\left\{\begin{aligned}
  &C(t)^\top P_2(t)\bar{D}_1(t)+P_2(t)\bar{B}_1(t)=0,\quad \ t\in[0,T],\\
  &C(t)^\top P_2(t)\bar{C}(t)+P_2(t)\bar{A}(t)=0,\quad \ t\in[0,T],\\
  &C(t)^\top P_2(t)\bar{D}_2(t)+P_2(t)\bar{B}_2(t)=0,\quad \ t\in[0,T],
  \end{aligned}\right.
\end{equation*}
where $P_2(\cdot)$ is the solution to the following Pseudo-Riccati equation:
\begin{equation}\left\{\begin{aligned}\label{P2-equation}
   \dot{P}_2(t)&=-P_2(t)\hat{A}_1(t)-\hat{A}_1(t)^\top P_2(t)-\hat{C}_1(t)^\top P_2(t)\hat{C}_1(t)\\
               &\quad-\hat{C}_2(t+\delta)^\top P_2(t+\delta)\hat{C}_2(t+\delta)-Q_2(t)-\bar{Q}_2(t+\delta)\\
               &\quad+\hat{C}_2(t+\delta)^\top P_2(t+\delta)\hat{D}(t+\delta)\Omega_2^{-1}(t)\hat{D}(t+\delta)^\top P_2(t+\delta)\hat{C}_2(t+\delta),\ t\in[0,T],\\
         P_2(T)&=G_2, P_2(t)=0,\ t\in(T,T+\delta],\\
    \Omega_2(t)&\triangleq R_2(t)+\bar{R}_2(t+\delta)+\hat{D}(t+\delta)^TP_2(t+\delta)\hat{D}(t+\delta)>0,\quad \forall\ t\in[0,T].
\end{aligned}\right.\end{equation}

Let $(\zeta_2(\cdot),\bar{\zeta}_2(\cdot))$ satisfy the following ABSDE:
\begin{equation}\left\{\begin{aligned}\label{zeta2-equation}
  d\zeta_2(t)&=\Big\{-\hat{A}_1(t)^\top\zeta_2(t)-\hat{C}_1(t)^\top\bar{\zeta}_2(t)+\big[\hat{C}_2(t+\delta)^\top P_2(t+\delta)\hat{D}(t+\delta)\Omega_2^{-1}(t)\hat{B}(t+\delta)^\top\\
             &\qquad-\hat{A}_2(t+\delta)^\top\big]\mathbb{E}^{\mathcal{F}_t}[\zeta_2(t+\delta)]+\big[\hat{C}_2(t+\delta)^\top P_2(t+\delta)\hat{D}(t+\delta)\Omega_2^{-1}(t)\hat{D}(t+\delta)^\top\\
             &\qquad-\hat{C}_2(t+\delta)^\top\big]\mathbb{E}^{\mathcal{F}_t}[\bar{\zeta}_2(t+\delta)]-\hat{C}_2(t+\delta)^\top P_2(t+\delta)\hat{D}(t+\delta)\Omega_2^{-1}(t)\big\{\hat{N}_1(t)^\top\xi(t)\\
             &\qquad+\mathbb{E}^{\mathcal{F}_t}[\hat{N}_2(t+\delta)^\top\xi(t+\delta)]\big\}+\big[-\hat{C}_2(t+\delta)^\top P_2(t+\delta)\hat{K}(t+\delta)\\
             &\qquad+\hat{C}_2(t+\delta)^\top P_2(t+\delta)\hat{D}(t+\delta)\Omega_2^{-1}(t)\hat{D}(t+\delta)^\top P_2(t+\delta)\hat{K}(t+\delta)\big]\mathbb{E}^{\mathcal{F}_t}[\zeta_1(t+\delta)]\\
             &\qquad+\big[\hat{C}_2(t+\delta)^\top P_2(t+\delta)\hat{D}(t+\delta)\Omega_2^{-1}(t)\hat{D}(t+\delta)^\top P_2(t+\delta)\hat{M}(t+\delta)\\
             &\qquad-\hat{C}_2(t+\delta)^\top P_2(t+\delta)\hat{M}(t+\delta)\big]\mathbb{E}^{\mathcal{F}_t}[\bar{\zeta}_1(t+\delta)]\Big\}dt+\bar{\zeta}_2(t)dW(t),\ t\in[0,T],\\
   \zeta_2(t)&=\bar{\zeta}_2(t)=0,\ t\in[T,T+\delta].
\end{aligned}\right.\end{equation}

\begin{Remark}\label{rem4.1}
Similar to (\ref{P1-equation}), (\ref{P2-equation}) is not a Riccati-type equation and can be solved step by step as a linear ODE, like (19) in \cite{XSZ18}. On the other hand, ABSDE (\ref{zeta2-equation}) has the unique solution $(\zeta_2(\cdot),\bar{\zeta}_2(\cdot))\in L_{\mathcal{F}}^2(\Omega;C([0,T];\textbf{R}^n))\times L_{\mathcal{F}}^2([0,T];\textbf{R}^n)$ (see \cite{PY09}).
\end{Remark}

\vspace{1mm}

It is easy to verify that
\begin{equation}\left\{\begin{aligned}\label{eq4.6}
  p_2(t)&=P_2(t)\bar{X}(t)+\zeta_2(t),\ t\in[0,T],\ P-a.s.\\
  q_2(t)&=P_2(t)\hat{C}_1(t)\bar{X}(t)+P_2(t)\hat{C}_2(t)\bar{X}(t-\delta)+P_2(t)\hat{D}(t)\bar{u}_2(t-\delta)\\
        &\quad+P_2(t)\hat{K}(t)\mathbb{E}^{\mathcal{F}_{t-\delta}}[\zeta_1(t)]+P_2(t)\hat{M}(t)\mathbb{E}^{\mathcal{F}_{t-\delta}}[\bar{\zeta}_1(t)]\\
        &\quad+P_2(t)\bar{D}_1(t)\eta_1(t-\delta)I_{[0,\delta]}(t)+\bar{\zeta}_2(t),\ t\in[0,T],\ P-a.s..
\end{aligned}\right.\end{equation}
Furthermore, with some computations, (\ref{u2-optimal}) becomes
\begin{equation}\begin{aligned}\label{u2-optimal-2}
  \bar{u}_2(t)=&-\Omega_2^{-1}(t)\Big\{\hat{D}(t+\delta)^\top P_2(t+\delta)\hat{C}_2(t+\delta)\bar{X}(t)+\hat{B}(t+\delta)^\top\mathbb{E}^{\mathcal{F}_t}[\zeta_2(t+\delta)]\\
               &\ +\hat{D}(t+\delta)^\top\mathbb{E}^{\mathcal{F}_t}[\bar{\zeta}_2(t+\delta)]+\hat{D}(t+\delta)^\top P_2(t+\delta)\hat{K}(t+\delta)\mathbb{E}^{\mathcal{F}_t}[\zeta_1(t+\delta)]\\
               &\ +\hat{D}(t+\delta)^\top P_2(t+\delta)\hat{M}(t+\delta)\mathbb{E}^{\mathcal{F}_t}[\bar{\zeta}_1(t+\delta)]-\hat{N}_1(t)^\top\xi(t)\\
               &\ -\mathbb{E}^{\mathcal{F}_t}\big[\hat{N}_2(t+\delta)^\top\xi(t+\delta)\big]\Big\},\quad a.e.\ t\in[0,T],\ \mathbb{P}\mbox{-}a.s.
\end{aligned}\end{equation}

Noting the representation of optimal control $\bar{u}_2(\cdot)$ in (\ref{u2-optimal-2}) is not satisfactory, since in order to determine $\xi(\cdot)$, we need to solve the adjoint equation (\ref{p2-equation}). Hence we have to know the information $\bar{X}(T)$, this is not the desired result. We expect to obtain a kind of feedback expression for the optimal control $u_2(\cdot)$ similar to (\ref{u1-feedback}). For this target, we stack the forward variables and the backward variables obtained in the optimization of the follower and the leader, respectively.

\vspace{1mm}

We denote
\begin{eqnarray*}\begin{aligned}
  &\phi(t)\triangleq  \begin{bmatrix}
  \xi(t)\\
  \bar{X}(t)
  \end{bmatrix},\quad \psi(t)\triangleq \begin{bmatrix}\zeta_1(t)\\
  \zeta_2(t)
  \end{bmatrix},\quad \bar{\psi}(t)\triangleq \begin{bmatrix}\bar{\zeta}_1(t)\\
  \bar{\zeta}_2(t)
  \end{bmatrix},\quad \mathcal{A}_1(t)\triangleq \begin{bmatrix}
  \hat{A}_1(t) & 0\\
  0 & \hat{A}_1(t)
  \end{bmatrix},\\
  &\mathcal{A}_2(t)\triangleq \begin{bmatrix}
  0 & \hat{F}(t)^\top P_2(t)+\hat{K}(t)^\top P_2(t)\hat{C}_1(t)\\
  0 & 0
  \end{bmatrix},\quad \mathcal{A}_3(t)\triangleq \begin{bmatrix}
  \tilde{A}_2(t) & \hat{K}(t)^\top P_2(t)\hat{C}_2(t)\\
  0 & \hat{A}_2(t)
  \end{bmatrix},\\
  &\mathcal{B}(t)\triangleq \begin{bmatrix}
  \hat{K}(t)^\top P_2(t)\hat{K}(t) & \hat{F}(t)^\top\\
  \hat{F}(t) & 0
  \end{bmatrix},\quad \mathcal{C}(t)\triangleq \begin{bmatrix}
  \hat{K}(t)^\top P_2(t)\hat{M}(t) & \hat{K}(t)^\top\\
  \hat{H}(t) & 0
  \end{bmatrix},\\
  &\bar{\mathcal{A}}_1(t)\triangleq \begin{bmatrix}
  \hat{C}_1(t) & 0\\
  0 & \hat{C}_1(t)
  \end{bmatrix},\quad \bar{\mathcal{A}}_2(t)\triangleq \begin{bmatrix}
  0 & \hat{H}(t)^\top P_2(t)+\hat{M}(t)^\top P_2(t)\hat{C}_1(t)\\
  0 & 0
  \end{bmatrix},\\
  &\bar{\mathcal{A}}_3(t)\triangleq \begin{bmatrix}
  \tilde{C}_2(t) & \hat{M}(t)^\top P_2(t)\hat{C}_2(t)\\
  0 & \hat{C}_2(t)
  \end{bmatrix},\quad \bar{\mathcal{C}}(t)\triangleq \begin{bmatrix}
  \hat{M}(t)^\top P_2(t)\hat{M}(t) & \hat{M}(t)^\top\\
  \hat{M}(t) & 0
  \end{bmatrix},\\
  &\mathcal{D}(t)\triangleq \begin{bmatrix}
  \hat{K}(t)^\top P_2(t)\hat{D}(t)\\
  \hat{B}(t)
  \end{bmatrix},\quad \bar{\mathcal{D}}(t)\triangleq \begin{bmatrix}
  \hat{M}(t)^\top P_2(t)\hat{D}(t)\\
  \hat{D}(t)
  \end{bmatrix},\\
  &\mathcal{G}_1(t)\triangleq \begin{bmatrix}
  \hat{N}_1(t)\\
  -\hat{C}_2(t+\delta)^\top P_2(t+\delta)\hat{D}(t+\delta)
  \end{bmatrix},\quad
  \mathcal{G}_2(t)\triangleq \begin{bmatrix}
  \hat{N}_2(t)\\
     0
  \end{bmatrix},\\
  &\mathcal{E}(t)\triangleq \begin{bmatrix}
  0 & 0\\
  0 & -\hat{C}_2(t+\delta)^\top P_2(t+\delta)\hat{D}(t+\delta)\Omega_2^{-1}(t)\hat{D}(t+\delta)^\top P_2(t+\delta)\hat{C}_2(t+\delta)
  \end{bmatrix},\\
  &\mathcal{M}(t)\triangleq \begin{bmatrix}
  \hat{K}(t)^\top P_2(t)\bar{D}_1(t)\eta_1(t-\delta)I_{[0,\delta]}(t)\\
  \bar{B}_1(t)\eta_1(t-\delta)I_{[0,\delta]}(t)
  \end{bmatrix},\quad \bar{\mathcal{M}}(t)\triangleq \begin{bmatrix}
  \hat{M}(t)^\top P_2(t)\bar{D}_1(t)\eta_1(t-\delta)I_{[0,\delta]}(t)\\
  \bar{D}_1(t)\eta_1(t-\delta)I_{[0,\delta]}(t)
  \end{bmatrix},\\
\end{aligned}\end{eqnarray*}
then we have
\begin{equation}\left\{\begin{aligned}\label{total-state-equation}
  d\phi(t)&=\Big\{\mathcal{A}_1(t)\phi(t)+\mathcal{A}_2(t)\mathbb{E}^{\mathcal{F}_{t-\delta}}[\phi(t)]+\mathcal{A}_3(t)\phi(t-\delta)+\mathcal{B}(t)\mathbb{E}^{\mathcal{F}_{t-\delta}}[\psi(t)]\\
          &\qquad+\mathcal{C}(t)\mathbb{E}^{\mathcal{F}_{t-\delta}}[\bar{\psi}(t)]+\mathcal{D}(t)\bar{u}_2(t-\delta)+\mathcal{M}(t)\Big\}dt\\
          &\quad+\Big\{\bar{\mathcal{A}}_1(t)\phi(t)+\bar{\mathcal{A}}_2(t)\mathbb{E}^{\mathcal{F}_{t-\delta}}[\phi(t)]+\bar{\mathcal{A}}_3(t)\phi(t-\delta)+\mathcal{C}(t)^\top\mathbb{E}^{\mathcal{F}_{t-\delta}}[\psi(t)]\\
          &\qquad+\bar{\mathcal{C}}(t)\mathbb{E}^{\mathcal{F}_{t-\delta}}[\bar{\psi}(t)]+\bar{\mathcal{D}}(t)\bar{u}_2(t-\delta)+\bar{\mathcal{M}}(t)\Big\}dW(t),\\
  d\psi(t)&=\Big\{\mathcal{E}(t)\phi(t)-\mathcal{A}_1(t)^\top\psi(t)-\mathcal{A}_2(t)^\top\mathbb{E}^{\mathcal{F}_{t-\delta}}[\psi(t)]-\mathcal{A}_3(t+\delta)^\top\mathbb{E}^{\mathcal{F}_t}[\psi(t+\delta)]\\
          &\qquad-\bar{\mathcal{A}_1}(t)^\top\bar{\psi}(t)-\bar{\mathcal{A}_2}(t)^\top\mathbb{E}^{\mathcal{F}_{t-\delta}}[\bar{\psi}(t)]-\bar{\mathcal{A}_3}(t+\delta)^\top\mathbb{E}^{\mathcal{F}_t}[\bar{\psi}(t+\delta)]\\
          &\qquad+\mathcal{G}_1(t)\bar{u}_2(t)+\mathcal{G}_2(t)\bar{u}_2(t-\delta)\Big\}dt+\bar{\psi}(t)dW(t),\ t\in[0,T],\\
  \phi(t)&=(0,\varphi(t)^\top)^\top,\ t\in[-\delta,0],\\
  \psi(t)&=\bar{\psi}(t)=(0,0)^\top,\ t\in[T,T+\delta],
\end{aligned}\right.\end{equation}
and (\ref{u2-optimal-2}) can be simplified as
\begin{equation}\begin{aligned}\label{u2-optimal-3}
  \bar{u}_2(t)&=-\Omega_2^{-1}(t)\Big\{-\mathcal{G}_1(t)^T\phi(t)+\mathcal{D}(t+\delta)^\top\mathbb{E}^{\mathcal{F}_t}[\psi(t+\delta)]+\bar{\mathcal{D}}(t+\delta)^\top\mathbb{E}^{\mathcal{F}_t}[\bar{\psi}(t+\delta)]\\
              &\qquad\qquad\quad\ -\mathcal{G}_2(t+\delta)^\top\mathbb{E}^{\mathcal{F}_t}[\phi(t+\delta)]\Big\},\quad a.e.\ t\in[0,T],\ \mathbb{P}\mbox{-}a.s.
\end{aligned}\end{equation}

\vspace{1mm}

To derive the feedback expression  of the optimal control $\bar{u}_2(\cdot)$, we will establish a nonhomogeneous relationship between $\phi(\cdot)$ and $\psi(\cdot)$ (see \cite{XSZ18}). Next basing on \textbf{(A1)}, \textbf{(A2)}, we also impose the following assumptions to the coefficients of (\ref{system equation}) and (\ref{cost}):
\begin{equation*}
  \textbf{(A3)}\left\{\begin{aligned}
  &I-\bar{D}_1(t)\Omega_1^{-1}(t-\delta)\bar{D}_1(t)^\top P_1(t)=0,\quad \ t\in[0,T],\\
  &\bar{A}(t)-\bar{B}_1(t)\Omega_1^{-1}(t-\delta)\bar{D}_1(t)^\top P_1(t)\bar{C}(t)=0,\quad \ t\in[0,T].
  \end{aligned}\right.
\end{equation*}

Hence
\begin{equation*}\left\{\begin{aligned}
  &\hat{A}_2(t)=\bar{A}(t),\ \hat{C}_2(t)=\bar{C}(t),\ \hat{D}(t)=\bar{D}_2(t),\ \tilde{A}_2(t)=\tilde{C}_2(t)=\hat{N}_1(t)=0,\quad t\in[0,\delta),\\
  &\tilde{A}_2(t)=\tilde{C}_2(t)=\hat{A}_2(t)\equiv\hat{C}_2(t)\equiv\hat{D}(t)\equiv\hat{N}_1(t)\equiv0,\quad t\in[\delta,T],\\
  &\mathcal{A}_3(t)\equiv\bar{\mathcal{A}_3}(t)\equiv\mathcal{E}(t)\equiv0,\quad \bar{\mathcal{D}}(t)\equiv\mathcal{G}_1(t)\equiv0,\quad t\in[\delta,T],\\
  &\mathcal{A}_3(t)=\begin{bmatrix}
  0 & \hat{K}(t)^TP_2(t)\bar{C}(t)\\
  0 & \bar{A}(t)
  \end{bmatrix},\quad \bar{\mathcal{A}_3}(t)=\begin{bmatrix}
  0 & \hat{M}(t)^TP_2(t)\bar{C}(t)\\
  0 & \bar{C}(t)
  \end{bmatrix},\quad t\in[0,\delta),\\
  &\mathcal{E}(t)\equiv0,\quad \bar{\mathcal{D}}(t)=\begin{bmatrix}
  \hat{M}(t)^TP_2(t)\bar{D}_2(t)\\
  \bar{D}_2(t)
  \end{bmatrix},\quad \mathcal{G}_1(t)\equiv0,\quad t\in[0,\delta).\\
\end{aligned}\right.\end{equation*}

Denote further
\begin{equation*}\left\{\begin{aligned}
  \Omega_3(t)&\triangleq R_2(t)+\bar{R}_2(t+\delta),\\
     \Xi_1(t)&\triangleq \mathcal{A}_2(t)+\mathcal{B}(t)L(t)+\mathcal{C}(t)[I-L(t)\bar{\mathcal{C}}(t)]^{-1}L(t)\big[\bar{\mathcal{A}_1}(t)+\bar{\mathcal{A}_2}(t)+\mathcal{C}(t)^\top L(t)]\\
             &\quad+\mathcal{D}(t)\Omega_3^{-1}(t-\delta)\big[\mathcal{G}_2(t)^\top-\mathcal{D}(t)^\top L(t)\big],\\
     \Xi_2(t)&\triangleq-\mathcal{B}(t)-\mathcal{C}(t)[I-L(t)\bar{\mathcal{C}}(t)]^{-1}L(t)\mathcal{C}(t)^\top+\mathcal{D}(t)\Omega_3^{-1}(t-\delta)\mathcal{D}(t)^\top,\\
     \Xi_3(t)&\triangleq[\mathcal{G}_2(t)-L(t)\mathcal{D}(t)]\Omega_3^{-1}(t-\delta)[\mathcal{G}_2(t)^\top-\mathcal{D}(t)^\top L(t)]-\big[\bar{\mathcal{A}_1}(t)+\bar{\mathcal{A}_2}(t)\\
             &\quad+\mathcal{C}(t)^\top L(t)\big]^\top[I-L(t)\bar{\mathcal{C}}(t)]^{-1}L(t)\big[\bar{\mathcal{A}_1}(t)+\bar{\mathcal{A}_2}(t)+\mathcal{C}(t)^\top L(t)\big]\\
             &\quad-\mathcal{A}_2(t)^\top L(t)-L(t)\mathcal{A}_2(t)-L(t)\mathcal{B}(t)L(t)+\bar{\mathcal{A}_1}(t)^\top L(t)\bar{\mathcal{A}_1}(t),
\end{aligned}\right.\end{equation*}
where $L(\cdot)$ is the solution to the following matrix equation:
\begin{equation}\left\{\begin{aligned}\label{L-equation}
  \dot{L}(t)&=-L(t)\mathcal{A}_1(t)-\mathcal{A}_1(t)^\top L(t)+\Pi(t,t+\delta)I_{[0,T-\delta]}(t)-\bar{\mathcal{A}_1}(t)^\top L(t)\bar{\mathcal{A}}_1(t),\ t\in[0,T],\\
        L(T)&=0,\\
         [I-L&(t)\bar{\mathcal{C}}(t)]^{-1}\ \mbox{exists},\ t\in[0,T],
\end{aligned}\right.\end{equation}
and $\Pi(\cdot,\cdot)$ satisfies
\begin{equation}\left\{\begin{aligned}\label{Pi-equation}
  -\frac{\partial\Pi(t,\theta)}{\partial t}&=\Pi(t,\theta)\mathcal{A}_1(t)+\mathcal{A}_1(t)^\top\Pi(t,\theta),\ \theta-\delta\leq t\leq \theta,\\
                                   \Pi(t,t)&=\int_t^{(t+\delta)\wedge T}\Pi(t,\theta)d\theta\Xi_1(t)+\Xi_1(t)^\top\int_t^{(t+\delta)\wedge T}\Pi(t,\theta)d\theta\\
  &+\int_t^{(t+\delta)\wedge T}\Pi(t,\theta)d\theta\Xi_2(t)\int_t^{(t+\delta)\wedge T}\Pi(t,\theta)d\theta+\Xi_3(t).
\end{aligned}\right.\end{equation}
Apparently $\Pi(\cdot,\cdot)$ is symmetric, hence so is $L(\cdot)$.

\begin{Remark}\label{rem4.3}
 Comparing with (25) in \cite{XSZ18}, the matrix equation (\ref{L-equation}) is very simple.
\end{Remark}

Then we have the nonhomogeneous relationship between $\psi(\cdot)$ and $\phi(\cdot)$ which is shown below.
\begin{mylem}\label{lem4.1}
Provided that the matrix equation $(\ref{L-equation})$ and $(\ref{Pi-equation})$ have the unique solutions, then it holds that
\begin{equation}\label{relationship}
  \psi(t)=L(t)\phi(t)-\int_t^{(t+\delta)\wedge T}\Pi(t,\theta)\hat{\phi}(t|\theta-\delta)d\theta,\quad t\in[\delta,T],
\end{equation}
where $\hat{\phi}(t|\theta-\delta)\triangleq\mathbb{E}^{\mathcal{F}_{\theta-\delta}}[\phi(t)]$.
\end{mylem}
\begin{proof}
Let
\begin{equation}\label{eq4.13}
  \Theta(t)\triangleq \psi(t)-L(t)\phi(t),\quad t\in[\delta,T],
\end{equation}
and suppose $d\Theta(t)=\Theta_1(t)dt+\Theta_2(t)dW(t)$. Applying It\^o's formula to (\ref{eq4.13}), we have
\begin{equation}\begin{aligned}\label{eq4.14}
  d\Theta(t)&=\Big\{-\mathcal{A}_1(t)^\top\psi(t)-\mathcal{A}_2(t)^\top\mathbb{E}^{\mathcal{F}_{t-\delta}}[\psi(t)]-\bar{\mathcal{A}_1}(t)^\top\bar{\psi}(t)
             -\bar{\mathcal{A}_2}(t)^\top\mathbb{E}^{\mathcal{F}_{t-\delta}}[\bar{\psi}(t)]\\
            &\qquad+\mathcal{G}_2(t)\bar{u}_2(t-\delta)+\mathcal{A}_1(t)^\top L(t)\phi(t)-\Pi(t,t+\delta)I_{[0,T-\delta]}(t)\phi(t)\\
            &\qquad+\bar{\mathcal{A}_1}(t)^\top L(t)\bar{\mathcal{A}}_1(t)\phi(t)-L(t)\mathcal{A}_2(t)\mathbb{E}^{\mathcal{F}_{t-\delta}}[\phi(t)]-L(t)\mathcal{B}(t)\mathbb{E}^{\mathcal{F}_{t-\delta}}[\psi(t)]\\
            &\qquad-L(t)\mathcal{C}(t)\mathbb{E}^{\mathcal{F}_{t-\delta}}[\bar{\psi}(t)]-L(t)\mathcal{D}(t)u_2(t-\delta)\Big\}dt\\
            &\quad+\Big\{\bar{\psi}(t)-L(t)\bar{\mathcal{A}}_1(t)\phi(t)-L(t)\bar{\mathcal{A}}_2(t)\mathbb{E}^{\mathcal{F}_{t-\delta}}[\phi(t)]\\
            &\qquad-L(t)\mathcal{C}(t)^\top\mathbb{E}^{\mathcal{F}_{t-\delta}}[\psi(t)]-L(t)\bar{\mathcal{C}}(t)\mathbb{E}^{\mathcal{F}_{t-\delta}}[\bar{\psi}(t)]\Big\}dW(t),\quad t\in[\delta,T].
\end{aligned}\end{equation}
Then we obtain
\begin{equation}\begin{aligned}\label{eq4.15}
  \bar{\psi}(t)&=\Theta_2(t)+L(t)\bar{\mathcal{A}}_1(t)\phi(t)+L(t)\bar{\mathcal{A}}_2(t)\mathbb{E}^{\mathcal{F}_{t-\delta}}[\phi(t)]+L(t)\mathcal{C}(t)^T\mathbb{E}^{\mathcal{F}_{t-\delta}}[\psi(t)]\\
               &\quad+L(t)\bar{\mathcal{C}}(t)\mathbb{E}^{\mathcal{F}_{t-\delta}}[\bar{\psi}(t)],\quad t\in[\delta,T].
\end{aligned}\end{equation}
It follows that
\begin{equation}\begin{aligned}\label{eq4.16}
  \mathbb{E}^{\mathcal{F}_{t-\delta}}[\bar{\psi}(t)]&=[I-L(t)\bar{\mathcal{C}}(t)]^{-1}\Big\{\mathbb{E}^{\mathcal{F}_{t-\delta}}[\Theta_2(t)]+L(t)\bar{\mathcal{A}}_1(t)\mathbb{E}^{\mathcal{F}_{t-\delta}}[\phi(t)]\\
                                                    &\quad+L(t)\bar{\mathcal{A}}_2(t)\mathbb{E}^{\mathcal{F}_{t-\delta}}[\phi(t)]+L(t)\mathcal{C}(t)^T\mathbb{E}^{\mathcal{F}_{t-\delta}}[\psi(t)]\Big\},\quad t\in[\delta,T].
\end{aligned}\end{equation}
Substituting (\ref{eq4.13}) into (\ref{u2-optimal-3}), we deduce
\begin{equation}\begin{aligned}\label{eq4.17}
  \bar{u}_2(t)&=-\Omega_3^{-1}(t)\Big\{\big[\mathcal{D}(t+\delta)^\top L(t+\delta)-\mathcal{G}_2(t+\delta)^\top\big]\mathbb{E}^{\mathcal{F}_t}[\phi(t+\delta)]\\
              &\quad\quad\quad\quad\quad\quad+\mathcal{D}(t+\delta)^T\mathbb{E}^{\mathcal{F}_t}[\Theta(t+\delta)]\Big\},\quad a.e.\ t\in[0,T],\ \mathbb{P}\mbox{-}a.s.
\end{aligned}\end{equation}
Substituting (\ref{eq4.16}) into (\ref{eq4.15}), we derive
\begin{equation}\begin{aligned}\label{eq4.18}
  \bar{\psi}(t)&=L(t)\bar{\mathcal{A}}_1(t)\phi(t)+\Big\{L(t)\bar{\mathcal{C}}(t)[I-L(t)\bar{\mathcal{C}}(t)]^{-1}L(t)\big[\bar{\mathcal{A}_1}(t)+\bar{\mathcal{A}_2}(t)+\mathcal{C}(t)^\top L(t)\big]\\
               &\quad+L(t)\bar{\mathcal{A}}_2(t)+L(t)\mathcal{C}(t)^\top L(t)\Big\}\mathbb{E}^{\mathcal{F}_{t-\delta}}[\phi(t)]+\Big\{L(t)\bar{\mathcal{C}}(t)[I-L(t)\bar{\mathcal{C}}(t)]^{-1}L(t)\mathcal{C}(t)^\top\\
               &\quad+L(t)\mathcal{C}(t)^\top\Big\}\mathbb{E}^{\mathcal{F}_{t-\delta}}[\Theta(t)]+\Theta_2(t)+L(t)\bar{\mathcal{C}}(t)[I-L(t)\bar{\mathcal{C}}(t)]^{-1}\mathbb{E}^{\mathcal{F}_{t-\delta}}[\Theta_2(t)],\quad t\in[\delta,T].
\end{aligned}\end{equation}
Hence we have
\begin{equation}\begin{aligned}\label{eq4.19}
  \Theta_1(t)&=\Big\{-\big[\bar{\mathcal{A}_2}(t)^\top+L(t)\mathcal{C}(t)\big][I-L(t)\bar{\mathcal{C}}(t)]^{-1}L(t)\big[\bar{\mathcal{A}_1}(t)+\bar{\mathcal{A}_2}(t)+\mathcal{C}(t)^\top L(t)\big]\\
             &\qquad-\mathcal{A}_2(t)^\top L(t)-L(t)\mathcal{A}_2(t)-L(t)\mathcal{B}(t)L(t)\Big\}\mathbb{E}^{\mathcal{F}_{t-\delta}}[\phi(t)]-\bar{\mathcal{A}_1}(t)^\top\bar{\psi}(t)\\
             &\quad-\Pi(t,t+\delta)I_{[0,T-\delta]}(t)\phi(t)+\bar{\mathcal{A}_1}(t)^\top L(t)\bar{\mathcal{A}}_1(t)\phi(t)-\Big\{[\bar{\mathcal{A}}_2(t)^\top\\
             &\qquad+L(t)\mathcal{C}(t)][I-L(t)\bar{\mathcal{C}}(t)]^{-1}L(t)\mathcal{C}(t)^\top+\mathcal{A}_2(t)^\top+L(t)\mathcal{B}(t)\Big\}\mathbb{E}^{\mathcal{F}_{t-\delta}}[\Theta(t)]\\
             &\quad-\mathcal{A}_1(t)^\top\Theta(t)+\big[\mathcal{G}_2(t)-L(t)\mathcal{D}(t)\big]\bar{u}_2(t-\delta)\\
             &\quad-\big[\bar{\mathcal{A}_2}(t)^\top+L(t)\mathcal{C}(t)\big][I-L(t)\bar{\mathcal{C}}(t)]^{-1}\mathbb{E}^{\mathcal{F}_{t-\delta}}[\Theta_2(t)],\quad t\in[\delta,T].
\end{aligned}\end{equation}
Plugging (\ref{eq4.17}) and (\ref{eq4.18}) into (\ref{eq4.19}), it follows that
\begin{equation}\begin{aligned}\label{eq4.20}
  \Theta_1(t)&=\Big\{-\big[\bar{\mathcal{A}_2}(t)^\top+L(t)\mathcal{C}(t)\big][I-L(t)\bar{\mathcal{C}}(t)]^{-1}L(t)\big[\bar{\mathcal{A}_1}(t)+\bar{\mathcal{A}_2}(t)+\mathcal{C}(t)^\top L(t)\big]\\
             &\qquad-\bar{\mathcal{A}_1}(t)^\top L(t)\bar{\mathcal{C}}(t)[I-L(t)\bar{\mathcal{C}}(t)]^{-1}L(t)\big[\bar{\mathcal{A}_1}(t)+\bar{\mathcal{A}_2}(t)+\mathcal{C}(t)^\top L(t)\big]\\
             &\qquad-\bar{\mathcal{A}_1}(t)^\top L(t)\bar{\mathcal{A}}_2(t)-\bar{\mathcal{A}_1}(t)^\top L(t)\mathcal{C}(t)^\top L(t)-\mathcal{A}_2(t)^\top L(t)-L(t)\mathcal{A}_2(t)\\
             &\qquad-L(t)\mathcal{B}(t)L(t)+\big[\mathcal{G}_2(t)-L(t)\mathcal{D}(t)\big]\Omega_3^{-1}(t-\delta)\big[\mathcal{G}_2(t)^\top-\mathcal{D}(t)^\top L(t)\big]\Big\}\mathbb{E}^{\mathcal{F}_{t-\delta}}[\phi(t)]\\
             &\quad+\Big\{-\bar{\mathcal{A}_1}(t)^\top[I-L(t)\bar{\mathcal{C}}(t)]^{-1}L(t)\mathcal{C}(t)^\top-\big[\bar{\mathcal{A}}_2(t)^\top+L(t)\mathcal{C}(t)\big][I-L(t)\bar{\mathcal{C}}(t)]^{-1}L(t)\mathcal{C}(t)^\top\\
             &\qquad-\mathcal{A}_2(t)^\top-L(t)\mathcal{B}(t)-\big[\mathcal{G}_2(t)-L(t)\mathcal{D}(t)\big]\Omega_3^{-1}(t-\delta)\mathcal{D}(t)^\top\Big\}\mathbb{E}^{\mathcal{F}_{t-\delta}}[\Theta(t)]\\
             &\quad-\Pi(t,t+\delta)I_{[0,T-\delta]}(t)\phi(t)-\mathcal{A}_1(t)^\top\Theta(t)-\bar{\mathcal{A}_1}(t)^\top\Theta_2(t)\\
             &\quad-\big[\bar{\mathcal{A}_1}(t)^\top L(t)\bar{\mathcal{C}}(t)+\bar{\mathcal{A}_2}(t)^\top+L(t)\mathcal{C}(t)\big][I-L(t)\bar{\mathcal{C}}(t)]^{-1}\mathbb{E}^{\mathcal{F}_{t-\delta}}[\Theta_2(t)],\quad t\in[\delta,T].
\end{aligned}\end{equation}
With some computations, it yields
\begin{equation}\begin{aligned}\label{eq4.21}
   &-\big[\bar{\mathcal{A}_2}(t)^\top+L(t)\mathcal{C}(t)\big][I-L(t)\bar{\mathcal{C}}(t)]^{-1}L(t)\big[\bar{\mathcal{A}_1}(t)+\bar{\mathcal{A}_2}(t)+\mathcal{C}(t)^\top L(t)\big]\\
   &-\bar{\mathcal{A}_1}(t)^\top L(t)\bar{\mathcal{C}}(t)[I-L(t)\bar{\mathcal{C}}(t)]^{-1}L(t)\big[\bar{\mathcal{A}_1}(t)+\bar{\mathcal{A}_2}(t)+\mathcal{C}(t)^\top L(t)\big]
    -\bar{\mathcal{A}_1}(t)^\top L(t)\bar{\mathcal{A}}_2(t)\\
   &-\bar{\mathcal{A}_1}(t)^\top L(t)\mathcal{C}(t)^\top L(t)=-\big[\bar{\mathcal{A}_1}(t)^\top+\bar{\mathcal{A}_2}(t)^\top+L(t)\mathcal{C}(t)][I-L(t)\bar{\mathcal{C}}(t)]^{-1}L(t)\big[\bar{\mathcal{A}_1}(t)\\
   &+\bar{\mathcal{A}_2}(t)+\mathcal{C}(t)^\top L(t)\big]+\bar{\mathcal{A}_1}(t)^\top L(t)\bar{\mathcal{A}}_1(t).
\end{aligned}\end{equation}
On the other hand, for $t\in[T-\delta,T]$, denote $\tilde{\Theta}_1(t)=-\int_t^T\Pi(t,\theta)\hat{\phi}(t|\theta-\delta)d\theta$, we can prove that $(\tilde{\Theta}_1(\cdot),0)$ is the solution to the BSDE
\begin{equation*}\left\{\begin{aligned}
  d\Theta(t)&=\Theta_1(t)dt+\Theta_2(t)dW(t),\quad t\in[T-\delta,T],\\
   \Theta(T)&=0.
\end{aligned}\right.\end{equation*}
Hence
\begin{equation*}
  \Theta(T-\delta)=\tilde{\Theta}_1(T-\delta)=-\int_{T-\delta}^T\Pi(T-\delta,\theta)\hat{\phi}(T-\delta|\theta-\delta)d\theta.
\end{equation*}
For $t\in[\delta,T-\delta]$, denote $\tilde{\Theta}_2(t)=-\int_t^{t+\delta}\Pi(t,\theta)\hat{\phi}(t|\theta-\delta)d\theta$, next we aim to prove $(\tilde{\Theta}_2(\cdot),0)$ is the solution to the BSDE
\begin{equation}\left\{\begin{aligned}\label{Theta}
        d\Theta(t)&=\Theta_1(t)dt+\Theta_2(t)dW(t),\ t\in[\delta,T-\delta],\\
  \Theta(T-\delta)&=-\int_{T-\delta}^T\Pi(T-\delta,\theta)\hat{\phi}(T-\delta|\theta-\delta)d\theta.
\end{aligned}\right.\end{equation}
Differentiating on $\tilde{\Theta}_2(t)$ with respect to $t$, we get
\begin{equation}\begin{aligned}\label{eq4.22}
  d\tilde{\Theta}_2(t)=&-\Pi(t,t+\delta)\phi(t)+\Pi(t,t)\hat{\phi}(t|t-\delta)-\int_t^{t+\delta}\dot{\Pi}(t,\theta)\hat{\phi}(t|\theta-\delta)d\theta\\
                       &-\int_t^{t+\delta}\Pi(t,\theta)\Big\{\mathcal{A}_1(t)\hat{\phi}(t|\theta-\delta)+\mathcal{A}_2(t)\hat{\phi}(t|t-\delta)+\mathcal{B}(t)\mathbb{E}^{\mathcal{F}_{t-\delta}}[\psi(t)]\\
                       &\qquad+\mathcal{C}(t)\mathbb{E}^{\mathcal{F}_{t-\delta}}[\bar{\psi}(t)]+\mathcal{D}(t)u_2(t-\delta)\Big\}d\theta,\quad t\in[\delta,T-\delta].
\end{aligned}\end{equation}
Noting
\begin{equation}\begin{aligned}\label{eq4.23}
  \mathbb{E}^{\mathcal{F}_{t-\delta}}[\tilde{\Theta}_2(t)]=-\int_t^{t+\delta}\Pi(t,\theta)d\theta\hat{\phi}(t|t-\delta),
\end{aligned}\end{equation}
and recalling (\ref{eq4.16}) and (\ref{eq4.17}), we have
\begin{equation}\begin{aligned}\label{eq4.24}
  &\quad\mathcal{B}(t)\mathbb{E}^{\mathcal{F}_{t-\delta}}[\psi(t)]+\mathcal{C}(t)\mathbb{E}^{\mathcal{F}_{t-\delta}}[\bar{\psi}(t)]+\mathcal{D}(t)u_2(t-\delta)\\
  &=\mathcal{B}(t)L(t)\mathbb{E}^{\mathcal{F}_{t-\delta}}[\phi(t)]+\mathcal{B}(t)\mathbb{E}^{\mathcal{F}_{t-\delta}}[\Theta(t)]+\mathcal{C}(t)[I-L(t)\bar{\mathcal{C}}(t)]^{-1}\Big\{\mathbb{E}^{\mathcal{F}_{t-\delta}}[\Theta_2(t)]\\
  &\quad+L(t)\bar{\mathcal{A}}_1(t)\mathbb{E}^{\mathcal{F}_{t-\delta}}[\phi(t)]+L(t)\bar{\mathcal{A}}_2(t)\mathbb{E}^{\mathcal{F}_{t-\delta}}[\phi(t)]+L(t)\mathcal{C}(t)^\top L(t)\mathbb{E}^{\mathcal{F}_{t-\delta}}[\phi(t)]\\
  &\quad+L(t)\mathcal{C}(t)^\top\mathbb{E}^{\mathcal{F}_{t-\delta}}[\Theta(t)]\Big\}-\mathcal{D}(t)\Omega_3^{-1}(t-\delta)\Big\{\big[\mathcal{D}(t)^TL(t)-\mathcal{G}_2(t)^T\big]\mathbb{E}^{\mathcal{F}_{t-\delta}}[\phi(t)]\\
  &\quad+\mathcal{D}(t)^T\mathbb{E}^{\mathcal{F}_{t-\delta}}[\Theta(t)]\Big\},\quad t\in[\delta,T].
\end{aligned}\end{equation}
Plugging (\ref{Pi-equation}), (\ref{eq4.23}) and (\ref{eq4.24}) into (\ref{eq4.22}), combing (\ref{eq4.20}) and (\ref{eq4.21}), it is easy to verify that $(\tilde{\Theta}_2(\cdot),0)$ is the solution to BSDE (\ref{Theta}). Thus we complete the proof.
\end{proof}

\vspace{1mm}

Now we give the necessary condition of the solvability for Problem \textbf{(L-DLQ)}:
\begin{mythm}\label{thm4.2}
Let \textbf{(A1)-(A3)} hold, assume $\bar{Q}_i(t)=\bar{R}_i(t)=0$ for $t\in [T,T+\delta]$, $i=1,2$. Suppose the matrix equation $(\ref{L-equation})$ and $(\ref{Pi-equation})$ have the unique solutions $L(\cdot)$ and $\Pi(\cdot,\cdot)$ respectively. Let $\bar{u}_2(\cdot)$ be the optimal control of the leader and $\bar{X}(\cdot)$ be the optimal state strategy for the Problem (\textbf{L-DLQ}), then $\bar{u}_2(\cdot)$ is of the following state feedback form:
\begin{equation}\begin{aligned}\label{u2-feedback}
   \bar{u}_2(t)=&-\Omega_3^{-1}(t)\Big\{\big[\mathcal{D}(t+\delta)^\top L(t+\delta)-\mathcal{G}_2(t+\delta)^\top\big]\\
                &\qquad\quad-\mathcal{D}(t+\delta)^\top\int_{t+\delta}^{(t+2\delta)\wedge T}\Pi(t+\delta,\theta)d\theta\Big\}\hat{\phi}(t+\delta|t),\quad a.e.\ t\in[0,T],\ \mathbb{P}\mbox{-}a.s.
\end{aligned}\end{equation}
Moreover, the optimal cost of the leader can be expressed as follows:
\begin{equation}\begin{aligned}\label{leaderer-optimal-cost}
  &\underset{u_2(\cdot)\in\mathcal{U}_2[0,T]}{\inf}J_2(\bar{u}_1(\cdot),u_2(\cdot))=J_2(\bar{u}_1(\cdot),\bar{u}_2(\cdot))=\mathbb{E}\bigg\{\big\langle\varphi(0),P_2(0)\varphi(0)+\zeta_2(0)\big\rangle\\
  &\quad+\int_{-\delta}^0\Big[\big\langle\varphi(t),\bar{A}(t+\delta)^\top p_2(t+\delta)+\bar{C}(t+\delta)^\top q_2(t+\delta)+\bar{Q}_2(t+\delta)\varphi(t)\big\rangle\\
  &\qquad\qquad+\big\langle\eta_1(t),\bar{B}_1(t+\delta)^\top p_2(t+\delta)+\bar{D}_1(t+\delta)^\top q_2(t+\delta)\big\rangle+\big\langle\eta_2(t),\bar{R}_2(t+\delta)\eta_2(t)\\
  &\qquad\qquad+\bar{B}_2(t+\delta)^\top p_2(t+\delta)+\bar{D}_2(t+\delta)^\top q_2(t+\delta)-\hat{N}_2(t+\delta)^\top\xi(t+\delta)\big\rangle\Big]dt\bigg\},
\end{aligned}\end{equation}
where $P_2(\cdot)$ satisfies (\ref{P2-equation}), $(\xi(\cdot),p_2(\cdot),q_2(\cdot))$ satisfies (\ref{p2-equation}) and $(\zeta_2(\cdot),\bar{\zeta}_2(\cdot))$ satisfies (\ref{zeta2-equation}).
\end{mythm}
\begin{proof}
The feedback form (\ref{u2-feedback}) can be obtained by (\ref{relationship}), (\ref{eq4.13}) and (\ref{eq4.17}). Thus we only need to prove (\ref{leaderer-optimal-cost}). Recall (\ref{leader state}) and (\ref{p2-equation}), applying It\^o's formula to $\big\langle p_2(\cdot),\bar{X}(\cdot)\rangle-\langle\zeta_1(\cdot),\xi(\cdot)\big\rangle$, we deduce
\begin{equation}\begin{aligned}\label{eq4.27}
  &J_2(\bar{u}_1(\cdot),\bar{u}_2(\cdot))=\mathbb{E}\big[\langle p_2(0),\varphi(0)\rangle\big]+\mathbb{E}\int_{-\delta}^0\Big[\big\langle\varphi(t),\bar{A}(t+\delta)^\top p_2(t+\delta)+\bar{C}(t+\delta)^\top q_2(t+\delta)\\
  &+\bar{Q}_2(t+\delta)\varphi(t)\big\rangle+\big\langle\eta_1(t),\bar{B}_1(t+\delta)^\top p_2(t+\delta)+\bar{D}_1(t+\delta)^\top q_2(t+\delta)\big\rangle+\big\langle\eta_2(t),\bar{R}_2(t+\delta)\eta_2(t)\\
  &+\bar{B}_2(t+\delta)^\top p_2(t+\delta)+\bar{D}_2(t+\delta)^\top q_2(t+\delta)-\hat{N}_2(t+\delta)^\top\xi(t+\delta)\big\rangle\Big] dt\\
  &+\mathbb{E}\int_0^T\big\langle\bar{u}_2(t),\big[R_2(t)+\bar{R}_2(t+\delta)\big]\bar{u}_2(t)+\mathbb{E}^{\mathcal{F}_t}\big[\hat{B}(t+\delta)^\top p_2(t+\delta)-\hat{N}_2(t+\delta)^T\xi(t+\delta)\big]\big\rangle dt.
\end{aligned}\end{equation}
Since $(\ref{u2-optimal})$ holds, we obtain
\begin{equation*}\begin{aligned}\label{eq4.28}
  &J_2(\bar{u}_1(\cdot),\bar{u}_2(\cdot))=\mathbb{E}\big[\langle p_2(0),\varphi(0)\rangle\big]+\mathbb{E}\int_{-\delta}^0\Big[\big\langle\varphi(t),\bar{A}(t+\delta)^\top p_2(t+\delta)+\bar{C}(t+\delta)^\top q_2(t+\delta)\\
  &\quad+\bar{Q}_2(t+\delta)\varphi(t)\big\rangle+\big\langle\eta_1(t),\bar{B}_1(t+\delta)^\top p_2(t+\delta)+\bar{D}_1(t+\delta)^\top q_2(t+\delta)\big\rangle+\big\langle\eta_2(t),\bar{R}_2(t+\delta)\eta_2(t)\\
  &\quad+\bar{B}_2(t+\delta)^\top p_2(t+\delta)+\bar{D}_2(t+\delta)^\top q_2(t+\delta)-\hat{N}_2(t+\delta)^\top\xi(t+\delta)\big\rangle\Big]dt\\
  &=\mathbb{E}\Big\{\big\langle\varphi(0),P_2(0)\varphi(0)+\zeta_2(0)\big\rangle+\int_{-\delta}^0\Big[\big\langle\varphi(t),\bar{A}(t+\delta)^\top p_2(t+\delta)+\bar{C}(t+\delta)^\top q_2(t+\delta)\\
  &\qquad+\bar{Q}_2(t+\delta)\varphi(t)\big\rangle+\big\langle\eta_1(t),\bar{B}_1(t+\delta)^\top p_2(t+\delta)+\bar{D}_1(t+\delta)^\top q_2(t+\delta)\big\rangle+\big\langle\eta_2(t),\bar{R}_2(t+\delta)\\
  &\qquad\eta_2(t)+\bar{B}_2(t+\delta)^\top p_2(t+\delta)+\bar{D}_2(t+\delta)^\top q_2(t+\delta)-\hat{N}_2(t+\delta)^\top\xi(t+\delta)\big\rangle\Big]dt\Big\},
\end{aligned}\end{equation*}
which completes the proof.
\end{proof}

\vspace{1mm}

Finally we summarize the above contents and state the main results for the linear quadratic Stackelberg differential game with time delay.

\begin{mythm}\label{thm4.3}
Let \textbf{(A1)-(A3)} hold, assume $\bar{Q}_i(t)=\bar{R}_i(t)=0$ for $t\in [T,T+\delta]$, $i=1,2$. Suppose $(\bar{u}_1(\cdot),\bar{u}_2(\cdot))$ is the optimal open-loop strategy and the matrix equation $(\ref{L-equation})$ and $(\ref{Pi-equation})$ have the unique solutions $L(\cdot)$ and $\Pi(\cdot,\cdot)$ respectively, then the optimal open-loop strategy is given by
\begin{equation}\begin{aligned}\label{strategy-1}
  \bar{u}_1(t-\delta)&=K_1^{\bar{u}_1}(t-\delta)\phi(t-\delta)+K_2^{\bar{u}_1}(t-\delta)\hat{\phi}(t|t-\delta),\quad a.e.\ t\in[\delta,T+\delta],\ \mathbb{P}\mbox{-}a.s.,
\end{aligned}\end{equation}
\begin{equation}\hspace{-4.6cm}\begin{aligned}\label{strategy-2}
  \bar{u}_2(t-\delta)&=K^{\bar{u}_2}(t-\delta)\hat{\phi}(t|t-\delta),\quad a.e.\ t\in[\delta,T+\delta],\ \mathbb{P}\mbox{-}a.s.,
\end{aligned}\end{equation}
where
\begin{equation*}\begin{aligned}
     \hat{\phi}(t|t-\delta)&\triangleq\mathbb{E}^{\mathcal{F}_{t-\delta}}[\phi(t)],\\
    K^{\bar{u}_2}(t-\delta)&\triangleq-\Omega_3^{-1}(t-\delta)\bigg\{\mathcal{D}(t)^\top L(t)-\mathcal{G}_2(t)^\top-\mathcal{D}(t)^\top\int_{t}^{(t+\delta)\wedge T}\Pi(t,\theta)d\theta\bigg\},\\
  K_1^{\bar{u}_1}(t-\delta)&\triangleq-\Omega_1^{-1}(t-\delta)\big[0,\bar{D}_1(t)^\top P_1(t)\bar{C}(t)\big],\\
\end{aligned}\end{equation*}
\begin{equation}\begin{aligned}\label{coefficient}
  K_2^{\bar{u}_1}(t-\delta)&\triangleq\Omega_1^{-1}(t-\delta)\bigg\{\bar{D}_1(t)^\top P_1(t)\bar{D}_2(t)\Omega_3^{-1}(t-\delta)\big[\mathcal{D}(t)^\top L(t)-\mathcal{G}_2(t)^\top\big]\\
                           &\qquad-\big[\bar{B}_1(t)^\top,0\big]L(t)-\big[\bar{D}_1(t)^\top,0\big][I-L(t)\bar{\mathcal{C}}(t)]^{-1}L(t)\big[\bar{\mathcal{A}_1}(t)+\bar{\mathcal{A}_2}(t)\\
                           &\qquad+\mathcal{C}(t)^\top L(t)\big]-\Big[\bar{D}_1(t)^\top P_1(t)\bar{D}_2(t)\Omega_3^{-1}(t-\delta)\mathcal{D}(t)^\top-\big[\bar{B}_1(t)^\top,0\big]\\
                           &\qquad-\big[\bar{D}_1(t)^\top,0\big][I-L(t)\bar{\mathcal{C}}(t)]^{-1}L(t)\mathcal{C}(t)^T\Big]\int_t^{(t+\delta)\wedge T}\Pi(t,\theta)d\theta\bigg\}.
\end{aligned}\end{equation}
\end{mythm}
\begin{proof}
(\ref{strategy-2}) is the same as (\ref{u2-feedback}). Thus we only need to prove (\ref{strategy-1}).
Substituting (\ref{strategy-2}) into (\ref{u1-feedback}), it follows that
\begin{equation}\begin{aligned}\label{eq4.32}
  \bar{u}_1(t-\delta)&=-\Omega_1^{-1}(t-\delta)\Big\{\bar{D}_1(t)^\top P_1(t)\bar{C}(t)\bar{X}(t-\delta)+\mathbb{E}^{\mathcal{F}_{t-\delta}}\big[\bar{B}_1(t)^T\zeta_1(t)\\
                     &\qquad+\bar{D}_1(t)^\top\bar{\zeta}_1(t)]+\bar{D}_1(t)^\top P_1(t)\bar{D}_2(t)\bar{u}_2(t-\delta)\Big\}\\
                     &=-\Omega_1^{-1}(t-\delta)\bigg\{\big[0,\bar{D}_1(t)^\top P_1(t)\bar{C}(t)\big]\phi(t-\delta)+\big[\bar{B}_1(t)^\top,0\big]\mathbb{E}^{\mathcal{F}_{t-\delta}}[\psi(t)]\\
                     &\qquad+\big[\bar{D}_1(t)^\top,0\big]\mathbb{E}^{\mathcal{F}_{t-\delta}}[\bar{\psi}(t)]-\bar{D}_1(t)^\top P_1(t)\bar{D}_2(t)\Omega_3^{-1}(t-\delta)\bigg\{\mathcal{D}(t)^\top L(t)\\
                     &\qquad-\mathcal{G}_2(t)^\top-\mathcal{D}(t)^\top\int_{t}^{(t+\delta)\wedge T}\Pi(t,\theta)d\theta\bigg\}\hat{\phi}(t|t-\delta)\bigg\}.
\end{aligned}\end{equation}
Plugging (\ref{relationship}) and (\ref{eq4.16}) into (\ref{eq4.32}), we can deduce (\ref{strategy-1}). Hence we complete the proof.
\end{proof}

\begin{Remark}\label{rem4.3}
From Theorem \ref{thm4.3}, to derive the state feedback expression of the open-loop Stackelberg strategy $(\bar{u}_1(\cdot),\bar{u}_2(\cdot))$, where we impose the assumptions \textbf{(A1)-(A3)} to the coefficients of (\ref{system equation}) and (\ref{cost}). Especially, in the one-dimensional case, we can obtain the explicit solutions to the Pseudo-Riccati equation (\ref{P1-equation}) and (\ref{P2-equation}) as follows:
\begin{equation*}\begin{aligned}
  &P_1(t)=G_1e^{\int_t^T(2A(s)+C^2(s))}ds+\int_t^Te^{\int_t^s(2A(r)+C^2(r))dr}\big[Q_1(s)+\bar{Q}_1(s+\delta)\big]ds,\\
  &P_2(t)=G_2e^{\int_t^T(2A(s)+C^2(s))ds}+\int_t^Te^{\int_t^s(2A(r)+C^2(r))dr}\big[Q_2(s)+\bar{Q}_2(s+\delta)\big]ds.
\end{aligned}\end{equation*}
Therefore \textbf{(A1)-(A3)} can be simplified in the following two cases:
\begin{enumerate}
 \item $P_1>0,\quad P_2\neq0.$
   \begin{equation*}\begin{aligned}
    &\bar{B}_1(t)+C(t)\bar{D}_1(t)=0,\quad\bar{A}(t)+C(t)\bar{C}(t)=0,\quad\bar{B}_2(t)+C(t)\bar{D}_2(t)=0,\quad \forall\ t\in[0,T],\\
    &\bar{D}_1\neq0,\quad\bar{R}_1(t)+R_1(t-\delta)=0,\quad R_2(t-\delta)+\bar{R}_2(t)>0,\quad \forall\ t\in[0,T].
   \end{aligned}\end{equation*}
 \item $P_1>0,\quad P_2=0.$
   \begin{equation*}\begin{aligned}
    &\bar{B}_1(t)+C(t)\bar{D}_1(t)=0,\quad\bar{A}(t)+C(t)\bar{C}(t)=0,\quad \forall\ t\in[0,T].\\
    &\bar{D}_1\neq0,\quad\bar{R}_1(t)+R_1(t-\delta)=0,\quad R_2(t-\delta)+\bar{R}_2(t)>0,\quad \forall\ t\in[0,T].
   \end{aligned}\end{equation*}
\end{enumerate}

\end{Remark}

\section{Applications}

In this section, we will display two examples, one example is a simple one-dimensional case and can be solved completely. Another example is related to a class of resource allocation problems, although the explicit solution can not be obtained, the numerical algorithm is given and some simulations are done.

\subsection{One-dimensional case}
Suppose
\begin{equation*}\begin{aligned}
  &A=\bar{A}=\bar{C}=0,\quad C=-1,\quad\bar{B}_1(\cdot)=\bar{D}_1(\cdot)\neq 0,\quad\bar{B}_2(\cdot)=\bar{D}_2(\cdot)\neq 0,\\
  &R_1(\cdot)+\bar{R}_1(\cdot+\delta)=0,\quad G_1=G_2=G>0,\\
  &Q_1(\cdot)+\bar{Q}_1(\cdot+\delta)=Q_2(\cdot)+\bar{Q}_2(\cdot+\delta)>0.
\end{aligned}\end{equation*}
Thu the linear controlled SDDE (\ref{system equation}) is reduced to the following:
\begin{equation}\label{ex1-system equation}\left\{\begin{aligned}
   dX(t)&=\big[\bar{B}_1(t)u_1(t-\delta)+\bar{B}_2(t)u_2(t-\delta)\big]dt\\
        &\quad+\big[-X(t)+\bar{D}_1(t)u_1(t-\delta)+\bar{D}_2(t)u_2(t-\delta)\big]dW(t),\ t\in[0,T],\\
  u_1(t)&=\eta_1(t),\ u_2(t)=\eta_2(t),\ t\in[-\delta,0],
\end{aligned}\right.\end{equation}
and the cost functionals of the leader and follower are as follows: For $i=1,2$,
\begin{equation}\begin{aligned}\label{ex-1cost}
  J_i(u_1(\cdot),u_2(\cdot))&=\mathbb{E}\bigg[\int_0^T\Big[X(t)^\top Q_i(t)X(t)+X(t-\delta)^\top\bar{Q}_i(t)X(t-\delta)+u_i(t)^\top R_i(t)u_i(t)\\
                            &\qquad\qquad\quad+u_i(t-\delta)^\top\bar{R}_i(t)u_i(t-\delta)]dt+X(T)^\top GX(T)\bigg].
\end{aligned}\end{equation}
It is obvious that the coefficients satisfies \textbf{(A1)-(A3)}, thus
\begin{equation*}
  \Omega_1(t)\equiv P_1(t+\delta)\bar{D}_1^2(t+\delta),\quad \Omega_2(t)=\Omega_3(t)\equiv R_2(t)+\bar{R}_2(t+\delta),
\end{equation*}
where $P_1(\cdot)$ is the solution to the following linear ODE:
\begin{equation*}\left\{\begin{aligned}
  \dot{P}_1(t)&=-P_1(t)-Q_1(t)-\bar{Q}_1(t+\delta),\ t\in[0,T],\\
        P_1(T)&=G.
\end{aligned}\right.\end{equation*}
It is easy to verify that $P_1(t)=Ge^{T-t}+\int_t^Te^{s-t}[Q_1(s)+\bar{Q}_1(s+\delta)]ds>0$, hence $\Omega_1>0$ and $\Omega_1^{-1}(t)=P_1^{-1}(t+\delta)\bar{D}_1^{-2}(t+\delta)$. Now (\ref{P2-equation}) reads
\begin{equation*}\left\{\begin{aligned}
  \dot{P}_2(t)&=-P_2(t)-Q_2(t)-\bar{Q}_2(t+\delta),\ t\in[0,T],\\
        P_2(T)&=G,
\end{aligned}\right.\end{equation*}
and it follows that $P_2(t)=Ge^{T-t}+\int_t^Te^{s-t}[Q_2(s)+\bar{Q}_2(s+\delta)]ds=P_1(t)>0$. For simplicity, assume $\bar{B}_1(t)=\bar{B}_2(t)=\bar{D}_1(t)=\bar{D}_2(t)=0$ for $t\in[0,\delta]$. In this case,
\begin{eqnarray*}\begin{aligned}
  &\mathcal{A}_1\equiv\mathcal{A}_2\equiv\mathcal{A}_3\equiv\bar{\mathcal{A}}_2\equiv\bar{\mathcal{A}}_3\equiv\mathcal{E}\equiv0,\quad\quad\quad\quad
  \mathcal{D}\equiv\bar{\mathcal{D}}\equiv\mathcal{G}_1\equiv\mathcal{G}_2\equiv\mathcal{M}\equiv\bar{\mathcal{M}}\equiv0,\\
  &\mathcal{B}(t)=\mathcal{C}(t)= \bar{\mathcal{C}}(t)=P_1^{-1}(t)I_{[\delta,T]}(t)\begin{bmatrix}
  1 & -1\\
  -1 & 0
  \end{bmatrix},\quad\bar{\mathcal{A}}_1(t)= \begin{bmatrix}
  -1 & 0\\
  0 & -1
  \end{bmatrix}.
\end{aligned}\end{eqnarray*}
Denote
\begin{equation*}\begin{aligned}
  L(t)\triangleq \begin{bmatrix}
  L_1(t) & L_2(t)\\
  L_2(t) & L_3(t)
  \end{bmatrix},\quad\quad\quad \Pi(t,\theta)\triangleq \begin{bmatrix}
  \Pi_1(t,\theta) & \Pi_2(t,\theta)\\
  \Pi_2(t,\theta) & \Pi_3(t,\theta)
  \end{bmatrix},
\end{aligned}\end{equation*}
and note that
\begin{equation*}\begin{aligned}
  &[I-L(t)\bar{\mathcal{C}}(t)]^{-1}=\begin{bmatrix}
  1-L_1(t)P_1^{-1}(t)+L_2(t)P_1^{-1}(t) & L_1(t)P_1^{-1}(t)\\
  -L_2(t)P_1^{-1}(t)+L_3(t)P_1^{-1}(t) & 1+L_2(t)P_1^{-1}(t)
  \end{bmatrix}^{-1}\\
  &=\frac{P_1(t)}{(P_1(t)+L_2(t))^2-L_1(t)(P_1(t)+L_3(t))}\begin{bmatrix}
  P_1(t)+L_2(t) & -L_1(t)\\
  L_2(t)-L_3(t) & P_1(t)+L_2(t)-L_1(t)
  \end{bmatrix},\ t\in[\delta,T].
\end{aligned}\end{equation*}
Hence we have (omitting $t$)
\begin{equation*}\begin{aligned}
  &[I-L\bar{\mathcal{C}}]^{-1}L\big[\bar{\mathcal{A}_1}+\bar{\mathcal{A}_2}+\mathcal{C}^\top L\big]\\
  &=\frac{P_1}{(P_1+L_2)^2-L_1(P_1+L_3)}\begin{bmatrix}
  P_1L_1 & P_1L_2+L^2_2-L_1L_3\\
  P_1L_2+L^2_2-L_1L_3 & P_1L_3+L^2_2-L_1L_3
  \end{bmatrix}\\
  &\quad\times\begin{bmatrix}
  -1+P_1^{-1}(L_1-L_2) & P_1^{-1}(L_2-L_3)\\
  -P_1^{-1}L_1 & -1-P_1^{-1}L_2
  \end{bmatrix},\ t\in[\delta,T].
\end{aligned}\end{equation*}
It follows that for $t\in[\delta,T]$,
\begin{equation*}\begin{aligned}
  \Xi_1(t)&=P_1^{-1}\begin{bmatrix}
  L_1-L_2 & L_2-L_3\\
  -L_1 & -L_2
  \end{bmatrix}+\frac{1}{(P_1+L_2)^2-L_1(P_1+L_3)}\begin{bmatrix}
  1 & -1\\
  -1 & 0
  \end{bmatrix}\\
  &\quad\times\begin{bmatrix}
  P_1L_1 & P_1L_2+L^2_2-L_1L_3\\
  P_1L_2+L^2_2-L_1L_3 & P_1L_3+L^2_2-L_1L_3
  \end{bmatrix}\begin{bmatrix}
  -1+P_1^{-1}(L_1-L_2) & P_1^{-1}(L_2-L_3)\\
  -P_1^{-1}L_1 & -1-P_1^{-1}L_2
  \end{bmatrix},\\
  \Xi_2(t)&=P_1^{-1}\begin{bmatrix}
  -1 & 1\\
  1 & 0
  \end{bmatrix}-\frac{P_1^{-1}}{(P_1+L_2)^2-L_1(P_1+L_3)} \begin{bmatrix}
  1 & -1\\
  -1 & 0\\
  \end{bmatrix}\\
  &\quad\times
  \begin{bmatrix}
  P_1+L_2 & -L_1\\
  L_2-L_3 & P_1+L_2-L_1
  \end{bmatrix}\begin{bmatrix}
  L_1-L_2 & -L_1\\
  L_2-L_3 & -L_2
  \end{bmatrix},
\end{aligned}\end{equation*}
\begin{equation*}\begin{aligned}
  \Xi_3(t)&=-P_1^{-1}\begin{bmatrix}
  L^2_1-2L_1L_2 & L_1L_2-L^2_2-L_1L_3\\
  L_1L_2-L^2_2-L_1L_3 & L^2_2-2L_2L_3
  \end{bmatrix}+\begin{bmatrix}
  L_1 & L_2\\
  L_2 & L_3
  \end{bmatrix}\\
  &\quad-\frac{P_1}{(P_1+L_2)^2-L_1(P_1+L_3)}\begin{bmatrix}
  -1+P_1^{-1}(L_1-L_2) & -P_1^{-1}L_1\\
  P_1^{-1}(L_2-L_3) & -1-P_1^{-1}L_2
  \end{bmatrix}\\
  &\quad\times\begin{bmatrix}
  P_1L_1 & P_1L_2+L^2_2-L_1L_3\\
  P_1L_2+L^2_2-L_1L_3 & P_1L_3+L^2_2-L_1L_3
  \end{bmatrix}\begin{bmatrix}
  -1+P_1^{-1}(L_1-L_2) & P_1^{-1}(L_2-L_3)\\
  -P_1^{-1}L_1 & -1-P_1^{-1}L_2
  \end{bmatrix}.
\end{aligned}\end{equation*}
Then (\ref{L-equation}) takes the following form:
\begin{equation}\left\{\begin{aligned}\label{eq5.3}
  \dot{L}(t)&=\Pi(t,t+\delta)I_{[0,T-\delta]}(t)-L(t),\ t\in[0,T],\\
        L(T)&=0,\\
        [I-L&(t)\bar{\mathcal{C}}(t)]^{-1}\ \mbox{exists},\ t\in[0,T].
\end{aligned}\right.\end{equation}
and $\Pi(\cdot,\cdot)$ satisfies
\begin{equation}\left\{\begin{aligned}\label{eq5.4}
  -\frac{\partial\Pi(t,\theta)}{\partial t}&=0,\\
                                   \Pi(t,t)&=\int_t^{(t+\delta)\wedge T}\Pi(t,\theta)d\theta\Xi_1(t)+\Xi_1(t)^\top\int_t^{(t+\delta)\wedge T}\Pi(t,\theta)d\theta\\
                                           &\quad+\int_t^{(t+\delta)\wedge T}\Pi(t,\theta)d\theta\Xi_2(t)\int_t^{(t+\delta)\wedge T}\Pi(t,\theta)d\theta+\Xi_3(t).
\end{aligned}\right.\end{equation}
From the above, we can easily see that even though in the simple one-dimensional case, the components $L_1(\cdot),L_2(\cdot)$ and $L_3(\cdot)$ in the matrix equation (\ref{eq5.3}) are heavily cross-coupled. Hence it is more complex to derive the analytical solution for the general case of (\ref{L-equation}) and (\ref{Pi-equation}). However, Our example is very special, coincidentally we find that $L(\cdot)=0$ and $\Pi(\cdot,\cdot)=0$ are the solutions to (\ref{eq5.3}) and (\ref{eq5.4}), respectively. Hence the optimal open-loop strategy is
\begin{equation*}
  \bar{u}_1(t)=\bar{u}_2(t)=0,\ a.e.\ t\in[0,T],\ \mathbb{P}\mbox{-}a.s.
\end{equation*}

\subsection{Resource allocation problem}

In this section a modified model for a dynamic research and development resource allocation problem under rivalry is considered (see Scherer \cite{F1967}). It is assumed that two firms, labeled as firm A and firm B, are competing with each other for a share of the market for a specific consumer goods. The size of the market is assumed to be fixed for the interval of consideration and it is assumed that it will not change with time nor will it be affected by the research and development effort of either firm. However, each firm's share of the market does depend on the quality of their product and in turn depends on the research and
development effort of each firm. It is further assumed that the product is of such a nature that all other costs, such as the cost of modifying plants, are negligible compared with the cost of research and development. For highly technical products, this assumption may be fairly reasonable.

Let $V_A(t)$ and $V_B(t)$ be the amounts of money invested in the research and development by firm A and firm B at time $t$, respectively. Let $x(\cdot)$ be a certain kind of measurement of technical gap between these two firms. The evolution of this gap is modeled by
\begin{equation}\label{gap}
  \dot{x}(t)=V_B^{\frac{1}{2}}(t-\delta)-\alpha V_A^{\frac{1}{2}}(t-\delta),
\end{equation}
where $\delta$ is the time delay, which represents the time required for the scientific study, i.e. if the firm invests in research at time $t-\delta$, he will get a return at time $t$. The multiplying factor $\alpha\geq 1$ accounts for the fact that it is easier for a developing firm, firm A, to catch up than for firm B, which is technically advanced, to innovate.

It is assumed that the shares of the market are equal when there is no difference in technical levels. When
there is a gap between technical levels of these two firms, a portion of the market of the firm lagging technically is temporarily taken over by its rival. This portion is assumed to be proportional to the square of the difference in technical levels. It is assumed that there is no permanent or lasting effect of this take-over of the market. In other words, the share of the market at each instant is determined by the technical levels at that instant only. Thus, the revenues of these two firms for finite horizon $[0,T]$ are
\begin{equation}\begin{aligned}\label{cost-1}
  &J_A=\int_0^T\bigg[\frac{V}{2x_0^2}(x^2(t)-x^2(t-\delta))+(V_A(t)-e^{\gamma\delta}V_A(t-\delta))\bigg]e^{-\gamma t}dt,\\
  &J_B=\int_0^T\bigg[-\frac{V}{2x_0^2}(x^2(t)-e^{\gamma\delta}x^2(t-\delta))+(V_B(t)-e^{\frac{1}{2}\gamma\delta}V_B(t-\delta))\bigg]e^{-\gamma t}dt,
\end{aligned}\end{equation}
where $x_0$ is some constant such that, when $x$ reaches $x_0$, the market is completely taken over by firm B, $V$ is the quasi-rent that is assumed to be a constant and $\gamma$ is the discount rate. $V_A(t)-e^{\gamma\delta}V_A(t-\delta)$, $V_B(t)-e^{\frac{1}{2}\gamma\delta}V_B(t-\delta)$ represent the difference of the scientific research for firm A, B in the time interval $[t-\delta,t]$, respectively. $e^{\gamma\delta}$, $e^{\frac{1}{2}\gamma\delta}$ are the investment weights, since the firm A lags technically, he should invest more money in the scientific research and hence the weight is more high. $\frac{x^2(t)-x^2(t-\delta)}{x_0^2}$ represents the technology gap improvement ratio at time $t$ comparing with time $t-\delta$, since the firm B has stronger technology strength, the weight $e^{\gamma\delta}$ is imposed on the measurement of technology gap improvement ratio for firm B. In other word, the firm B has to improve more technology strength to achieve the same technology gap improvement ratio as the firm A. There is no doubt that both companies want to minimize scientific research investment in order to obtain more technological improvements, i.e. the firm A wants to minimize $\frac{x^2(t)-x^2(t-\delta)}{x_0^2}$ while minimize $V_A(t)-e^{\gamma\delta}V_A(t-\delta)$, and the firm B wants to maximize $\frac{x^2(t)-e^{\gamma\delta}x^2(t-\delta)}{x_0^2}$ while minimize $V_B(t)-e^{\frac{1}{2}\gamma\delta}V_B(t-\delta)$.

The above problem can be converted to a standard linear quadratic Stackelberg differential game with time delay by the following change of variables:
\begin{equation*}
  X(t)=e^{-\frac{\gamma t}{2}}x(t),\quad u_1(t)=V_A^{\frac{1}{2}}(t)e^{-\frac{\gamma}{2} t},\quad u_2(t)=V_B^{\frac{1}{2}}(t)e^{-\frac{\gamma}{2} t},\quad Q=\frac{V}{2x_0^2},\quad \bar{Q}=\frac{e^{-\gamma\delta}V}{2x_0^2},
\end{equation*}
then (\ref{gap})-(\ref{cost-1}) become
\begin{equation}\left\{\begin{aligned}\label{modify-deter}
  \dot{X}(t)=&-\frac{\gamma}{2}X(t)-\alpha e^{-\frac{\gamma\delta}{2}}u_1(t-\delta)+e^{-\frac{\gamma\delta}{2}}u_2(t-\delta),\ t\in[0,T],\\
         J_A=&\int_0^T\bigg[QX^2(t)-\bar{Q}X^2(t-\delta)+u_1^2(t)-u_1^2(t-\delta)\bigg]dt,\\
         J_B=&\int_0^T\bigg[-QX^2(t)+QX^2(t-\delta)+u_2^2(t)-e^{-\frac{\gamma\delta}{2}}u_2^2(t-\delta)\bigg]dt.
\end{aligned}\right.\end{equation}

There must be some uncertain factors affecting the scientific research output, so we extend the above model  to the stochastic case:
\begin{equation}\left\{\begin{aligned}\label{modify-stochas}
  dX(t)&=\Big[-\frac{\gamma}{2}X(t)-\alpha e^{-\frac{\gamma\delta}{2}}u_1(t-\delta)+e^{-\frac{\gamma\delta}{2}}u_2(t-\delta)\Big]dt\\
       &\quad+\big[CX(t)+\bar{D}_1u_1(t-\delta)+\bar{D}_2u_2(t-\delta)\big]dW(t),\ t\in[0,T],\\
    J_A&=\mathbb{E}\int_0^T\bigg[QX^2(t)-\bar{Q}X^2(t-\delta)+u_1^2(t)-u_1^2(t-\delta)\bigg]dt,\\
    J_B&=\mathbb{E}\int_0^T\bigg[-QX^2(t)+QX^2(t-\delta)+u_2^2(t)-e^{-\frac{\gamma\delta}{2}}u_2^2(t-\delta)\bigg]dt,
\end{aligned}\right.\end{equation}
where $C,\bar{D}_1,\bar{D}_2$ are all constants satisfying
\begin{equation*}
  \bar{D}_1\neq0,\quad\bar{B}_1+C\bar{D}_1=0,\quad \bar{B}_
  2+C\bar{D}_2\neq0.
\end{equation*}

Now we apply the theoretical results in the above two sections to seek the optimal open-loop strategy. The coefficients take the following values:
\begin{equation*}\begin{aligned}
  &A(t)=-\frac{\gamma}{2},\quad \bar{B}_1(t)=-\alpha e^{-\frac{\gamma\delta}{2}},\quad \bar{B}_2(t)=e^{-\frac{\gamma\delta}{2}},\quad \bar{A}(t)=\bar{C}(t)=0,\\
  &Q_1(t)=Q,\quad \bar{Q}_1(t)=-\bar{Q},\quad R_1(t)=1,\quad \bar{R}_1(t)=-1,\quad G_1=0,\\
  &Q_2(t)=-Q,\quad \bar{Q}_2(t)=Q,\quad R_2(t)=1,\quad \bar{R}_2(t)=-e^{-\frac{\gamma\delta}{2}},\quad G_2=0.
\end{aligned}\end{equation*}

Apparently the coefficients satisfies \textbf{(A1)-(A3)}, and now
\begin{equation*}
  \Omega_1(t)=P_1(t+\delta)\bar{D}_1^2(t+\delta),\quad \Omega_2(t)=\Omega_3(t)=1-e^{-\frac{\gamma\delta}{2}},
\end{equation*}
where $P_1(\cdot)$ is the solution the following linear ODE:
\begin{equation}\left\{\begin{aligned}\label{ex2-P1}
  \dot{P}_1(t)&=(\gamma-C^2)P_1(t)-Q+\bar{Q},\ t\in[0,T],\\
        P_1(T)&=0.
\end{aligned}\right.\end{equation}
It is easy to verify that $P_1(t)=(Q-\bar{Q})\int_t^Te^{(C^2-\gamma)(s-t)}ds>0$, hence $\Omega_1>0$ and $\Omega_1^{-1}(t)=P_1^{-1}(t+\delta)\bar{D}_1^{-2}(t+\delta)$. Now (\ref{P2-equation}) reads
\begin{equation}\left\{\begin{aligned}\label{ex2-P2}
  \dot{P}_2(t)&=(\gamma-C^2)P_2(t),\ t\in[0,T],\\
        P_2(T)&=0,
\end{aligned}\right.\end{equation}
and it follows that $P_2(t)\equiv0$. In this case,
\begin{eqnarray*}\begin{aligned}
  &\mathcal{A}_2\equiv\mathcal{A}_3\equiv\bar{\mathcal{A}}_2\equiv\bar{\mathcal{A}}_3\equiv\mathcal{E}\equiv0,\quad\quad\quad\quad
  \mathcal{G}_1\equiv0,\quad\quad\quad\quad
  \mathcal{A}_1(t)=-\frac{\gamma}{2}\begin{bmatrix}
  1 & 0\\
  0 & 1
  \end{bmatrix},\\
  &\mathcal{B}(t)=-C^2P_1^{-1}(t)I_{[\delta,T]}(t)\begin{bmatrix}
  0 & 1\\
  1 & 0
  \end{bmatrix},\quad\quad\quad\quad \mathcal{C}(t)=CP_1^{-1}(t)I_{[\delta,T]}(t)\begin{bmatrix}
  0 & 1\\
  1 & 0
  \end{bmatrix},\\
  &\bar{\mathcal{A}}_1(t)= \begin{bmatrix}
  C & 0\\
  0 & C
  \end{bmatrix},\quad\quad\quad\quad\quad\quad\bar{\mathcal{C}}(t)=-P_1^{-1}(t)I_{[\delta,T]}(t)\begin{bmatrix}
  0 & 1\\
  1 & 0
  \end{bmatrix},\\
  &\mathcal{D}(t)= \begin{bmatrix}
  0 \\
  \bar{B}_2+C\bar{D}_2I_{[\delta,T]}(t)
  \end{bmatrix},\quad\quad\quad\quad\quad\quad\mathcal{G}_2(t)= \begin{bmatrix}
  -P_1(t)(\bar{B}_2+C\bar{D}_2) \\
  0
  \end{bmatrix}.
\end{aligned}\end{eqnarray*}
Denote
\begin{equation*}\begin{aligned}
  L(t)\triangleq \begin{bmatrix}
  L_1(t) & L_2(t)\\
  L_2(t) & L_3(t)
  \end{bmatrix},\quad\quad\quad \Pi(t,\theta)\triangleq \begin{bmatrix}
  \Pi_1(t,\theta) & \Pi_2(t,\theta)\\
  \Pi_2(t,\theta) & \Pi_3(t,\theta)
  \end{bmatrix}.
\end{aligned}\end{equation*}
Similar to Section 5.1, with some computations, we can get for $t\in[\delta,T]$,
\begin{equation*}\begin{aligned}
  &[I-L(t)\bar{\mathcal{C}}(t)]^{-1}\\
  &=\frac{P_1(t)}{L_1(t)L_3(t)-(P_1(t)+L_2(t))^2}\begin{bmatrix}
  -P_1(t)-L_2(t) & L_1(t)\\
  L_3(t) & -P_1(t)-L_2(t)
  \end{bmatrix}.
\end{aligned}\end{equation*}
It follows that for $t\in[\delta,T]$ (omitting $t$)
\begin{equation*}\begin{aligned}
  \Xi_1(t)&=-C^2P_1^{-1}\begin{bmatrix}
  L_2 & L_3\\
  L_1 & L_2
  \end{bmatrix}+\frac{C^2P_1^{-1}}{L_1L_3-(P_1+L_2)^2}\begin{bmatrix}
  L_3 & -P_1-L_2\\
  -P_1-L_2 & L_1
  \end{bmatrix}\\
  &\quad\times\begin{bmatrix}
  L_1 & L_2\\
  L_2 & L_3
  \end{bmatrix}\begin{bmatrix}
  P_1+L_2 & L_3\\
  L_1 & P_1+L_2
  \end{bmatrix}-[1-e^{-\frac{\gamma\delta}{2}}]^{-1}(\bar{B}_2+C\bar{D}_2)^2\begin{bmatrix}
  0 & 0\\
  P_1+L_2 & L_3
  \end{bmatrix},\\
  \Xi_2(t)&=C^2P_1^{-1}\begin{bmatrix}
  0 & 1\\
  1 & 0
  \end{bmatrix}-\frac{C^2P_1^{-1}}{L_1L_3-(P_1+L_2)^2} \begin{bmatrix}
  L_3 & -P_1-L_2\\
  -P_1-L_2 & L_1\\
  \end{bmatrix}\\
  &\quad\times
  \begin{bmatrix}
  L_2 & L_1\\
  L_3 & L_2
  \end{bmatrix}+[1-e^{-\frac{\gamma\delta}{2}}]^{-1}(\bar{B}_2+C\bar{D}_2)^2\begin{bmatrix}
  0 & 0\\
  0 & 1
  \end{bmatrix},\\
  \Xi_3(t)&=C^2P_1^{-1}\begin{bmatrix}
  2L_1L_2 & L^2_2+L_1L_3\\
  L^2_2+L_1L_3 & 2L_2L_3
  \end{bmatrix}+[1-e^{-\frac{\gamma\delta}{2}}]^{-1}(\bar{B}_2+C\bar{D}_2)^2\\
  &\quad\times\begin{bmatrix}
  (P_1+L_2)^2 & L_3(P_1+L_2)\\
  L_3(P_1+L_2) & L_3^2
  \end{bmatrix}+C^2\begin{bmatrix}
  L_1 & L_2\\
  L_2 & L_3
  \end{bmatrix}\\
  &\quad-\frac{C^2P_1^{-1}}{L_1L_3-(P_1+L_2)^2}\begin{bmatrix}
  P_1+L_2 & L_1\\
  L_3 & P_1+L_2
  \end{bmatrix}\\
  &\quad\times\begin{bmatrix}
  -P_1-L_2 & L_1\\
  L_3 & -P_1-L_2
  \end{bmatrix}\begin{bmatrix}
  L_1 & L_2\\
  L_2 & L_3
  \end{bmatrix}\begin{bmatrix}
  P_1+L_2 & L_3\\
  L_1 & P_1+L_2
  \end{bmatrix}.
\end{aligned}\end{equation*}
Then (\ref{L-equation}) becomes:
\begin{equation}\left\{\begin{aligned}\label{ex2-L}
  \dot{L}(t)&=(\gamma-C^2)L(t)+\Pi(t,t+\delta)I_{[0,T-\delta]}(t),\ t\in[\delta,T],\\
        L(T)&=0,\\
        [I-L&(t)\bar{\mathcal{C}}(t)]^{-1}\ \mbox{exists},\ t\in[0,T].
\end{aligned}\right.\end{equation}
and $\Pi(\cdot,\cdot)$ satisfies
\begin{equation}\left\{\begin{aligned}\label{ex2-Pi}
  -\frac{\partial\Pi(t,\theta)}{\partial t}=&-\gamma\Pi(t,\theta),\\
                                   \Pi(t,t)=&\int_t^{(t+\delta)\wedge T}\Pi(t,\theta)d\theta\Xi_1(t)+\Xi_1(t)^\top\int_t^{(t+\delta)\wedge T}\Pi(t,\theta)d\theta\\
                                            &+\int_t^{(t+\delta)\wedge T}\Pi(t,\theta)d\theta\Xi_2(t)\int_t^{(t+\delta)\wedge T}\Pi(t,\theta)d\theta+\Xi_3(t).
\end{aligned}\right.\end{equation}
Because the components $L_1(\cdot),L_2(\cdot)$ and $L_3(\cdot)$ in the matrix equation are coupled, hence we can not give the analytical solution to (\ref{ex2-L}).

Moreover, the optimal open-loop strategy is as follows:
\begin{equation*}\begin{aligned}
  \bar{u}_1(t-\delta)&=K_1(t-\delta)\hat{\phi}(t|t-\delta),\ a.e.\ t\in[\delta,T+\delta],\ \mathbb{P}\mbox{-}a.s.,
\end{aligned}\end{equation*}
\begin{equation*}\begin{aligned}
  \bar{u}_2(t-\delta)&=K_2(t-\delta)\hat{\phi}(t|t-\delta),\ a.e.\ t\in[\delta,T+\delta],\ \mathbb{P}\mbox{-}a.s.,
\end{aligned}\end{equation*}
where
\begin{equation*}\begin{aligned}
  K_1(t-\delta)&\triangleq\Omega_1^{-1}(t-\delta)\bigg\{\bar{D}_1\bar{D}_2P_1(t)\Omega_3^{-1}(t-\delta)(\bar{B}_2+C\bar{D}_2)\big[L_2(t)+P_1(t),L_3(t)\big]\\
               &\quad-\big[\bar{B}_1L_1(t),\bar{B}_1L_2(t)\big]-\frac{C\bar{D}_1}{L_1(t)L_3(t)-(P_1(t)+L_2(t))^2}\big[-P_1(t)-L_2(t),L_1(t)\big]\\
               &\qquad\times\begin{bmatrix}
                               L_1(t) & L_2(t)\\
                               L_2(t) & L_3(t)
                            \end{bmatrix}\begin{bmatrix}
                               P_1(t)+L_2(t) & L_3(t)\\
                               L_1(t) & P_1(t)+L_2(t)
                            \end{bmatrix}\\
               &\quad-\bigg[\bar{D}_1\bar{D}_2P_1(t)\Omega_3^{-1}(t-\delta)\big[0,\bar{B}_2+C\bar{D}_2\big]-\big[\bar{B}_1,0\big]\\
               &\qquad-\frac{C\bar{D}_1}{L_1(t)L_3(t)-(P_1(t)+L_2(t))^2}\big[-P_1(t)-L_2(t),L_1(t)\big]\begin{bmatrix}
                                                                                                          L_2(t) & L_1(t)\\
                                                                                                          L_3(t) & L_2(t)
                                                                                                       \end{bmatrix}\bigg]\\
               &\qquad\times\int_t^{(t+\delta)\wedge T}\Pi(t,\theta)d\theta\bigg\},\\
  K_2(t-\delta)&\triangleq-(\bar{B}_2+C\bar{D}_2)\Omega_3^{-1}(t-\delta)\begin{bmatrix}
                                                                            L_2(t)+P_1(t)-\int_{t}^{(t+\delta)\wedge T}\Pi_2(t,\theta)d\theta\\
                                                                            L_3(t)-\int_{t}^{(t+\delta)\wedge T}\Pi_3(t,\theta)d\theta
                                                                        \end{bmatrix}^\top.
\end{aligned}\end{equation*}

Finally, we give some simulation results. The following algorithm schemes are designed to simulate the numerical solution to (\ref{ex2-L}):

\begin{figure}[H]
\centering
\tikzstyle{startstop} = [rectangle,rounded corners, minimum width=1cm,minimum height=0.3cm,text centered, draw=black,fill=red!30]
\tikzstyle{io} = [trapezium, trapezium left angle = 70,trapezium right angle=110,minimum width=2cm,minimum height=0.3cm,text centered,draw=black,fill=blue!30]
\tikzstyle{process} = [rectangle,minimum width=2cm,minimum height=0.3cm,text centered,text width =1.5cm,draw=black,fill=orange!30]
\tikzstyle{process1} = [rectangle,minimum width=2cm,minimum height=0.3cm,text centered,text width =2cm,draw=black,fill=orange!30]
\tikzstyle{process2} = [rectangle,minimum width=2cm,minimum height=0.3cm,text centered,text width =7cm,draw=black,fill=orange!30]
\tikzstyle{process3} = [rectangle,minimum width=2cm,minimum height=0.3cm,text centered,text width =5cm,draw=black,fill=orange!30]
\tikzstyle{process4} = [rectangle,minimum width=2cm,minimum height=0.3cm,text centered,text width =7cm,draw=black,fill=orange!30]
\tikzstyle{decision} = [diamond,minimum width=2cm,minimum height=0.3cm,text centered,draw=black,fill=green!30]
\tikzstyle{arrow} = [thick,->,>=stealth]
\tikzstyle{arrow1} = [thick,-,>=stealth]
\begin{tikzpicture}[node distance=0.8cm]
\node (start) [startstop] {\scriptsize Start};
\node (input1) [io,below of=start] {\tiny input $T$, $\delta$, $N$};
\node (process1) [process,below of=input1] {\tiny$\Delta=\frac{T}{N+1}$};
\node (process2) [process,below of=process1] {\tiny$d=\frac{\delta}{\Delta}$};
\node (process3) [process,below of=process2] {\tiny$k=N+1$};
\node (decision1) [decision,below of=process3,yshift=-1cm] {\tiny$k>N-d+1$};
\node (process4a) [process,left of=decision1,xshift=-3.3cm] {\tiny$L(k)=0$};
\node (process5) [process1,below of=process4a] {\tiny$\Pi(k,k+d)=0$};
\node (process6) [process,below of=process5] {\tiny$k=k-1$};
\node (decision4b) [decision,below of =decision1,yshift=-1.5cm] {\tiny$k\geq 0$};
\node (process7) [process2,below of=decision4b,yshift=-0.5cm] {\tiny$L(k)=L(k+1)+(C^2-\gamma)\Delta L(k+1)-\Delta\Pi(k+1,k+1+d)$};
\node (process12) [process,below of =process7] {\tiny$j=k+2$};
\node (process13) [process,below of=process12] {\tiny$mid\Pi=0$};
\node (decision3) [decision,below of =process13,yshift=-1.5cm] {\tiny
$j\leq \min(N+1,k+1+d)$};
\node (process8) [process4,below of=decision3,yshift=-2cm] {\tiny$\Pi(k+1,k+1)=\Delta mid\Pi\Xi_1(k+1)+\Delta\Xi_1(k+1)mid\Pi+\Delta^2mid\Pi\Xi_2(k+1)mid\Pi+\Xi_3(k+1)$};
\node (process9) [process,below of=process8,yshift=-0.2cm] {\tiny$i=1$};
\node (process14) [process3,left of=decision3,xshift=-5.3cm] {\tiny$mid\Pi=mid\Pi+\Pi(k+1,j)$};
\node (process15) [process,left of=process14,xshift=-3.3cm] {\tiny$j=j+1$};
\node (decision2) [decision,below of =process9,yshift=-1.5cm] {\tiny$i\leq \min(d,N+1-k)$};
\node (process10) [process3,left of=decision2,xshift=-4.2cm] {\tiny$\Pi(k,k+i)=(1-\gamma\Delta)\Pi(k+1,k+i)$};
\node (process11) [process,left of=process10,xshift=-3.3cm] {\tiny$i=i+1$};
\node (process16) [process,right of=decision2,xshift=3.3cm] {\tiny$k=k-1$};
\node (stop) [startstop,left of=decision4b,xshift=-1.6cm] {\scriptsize Stop};
\draw [arrow] (start) -- (input1);
\draw [arrow] (input1) -- (process1);
\draw [arrow] (process1) -- (process2);
\draw [arrow] (process2) -- (process3);
\draw [arrow] (process3) -- (decision1);
\draw [arrow] (decision1) -- node[anchor=south] {yes} (process4a);
\draw [arrow] (process4a) -- (process5);
\draw [arrow] (process5) -- (process6);
\draw [arrow1] (process6) -- (-6,-6.6);
\draw [arrow1] (-6,-6.6) -- (-6,-4);
\draw [arrow] (-6,-4) -- (-0.2,-4);
\draw [arrow] (decision1) -- node[anchor=west] {no} (decision4b);
\draw [arrow] (decision4b) -- node[anchor=south] {no}(stop);
\draw [arrow] (decision4b) -- node[anchor=west] {yes} (process7);
\draw [arrow] (process7) -- (process12);
\draw [arrow] (process8) -- (process9);
\draw [arrow] (process9) -- (decision2);
\draw [arrow] (decision2) -- node[anchor=south]{yes} (process10);
\draw [arrow] (process10) -- (process11);
\draw [arrow1] (process11) -- (-9.1,-16.8);
\draw [arrow] (-9.1,-16.8) -- (-0.2,-16.8);
\draw [arrow] (decision2) -- node[anchor=south] {no} (process16);
\draw [arrow1] (process16) -- (4,-4);
\draw [arrow] (4,-4) -- (0.2,-4);
\draw [arrow] (process12) -- (process13);
\draw [arrow] (process13) -- (decision3);
\draw [arrow] (decision3) -- node[anchor=west] {no} (process8);
\draw [arrow] (decision3) -- node[anchor=south] {yes} (process14);
\draw [arrow] (process14) -- (process15);
\draw [arrow1] (process15) -- (-10.2,-10.7);
\draw [arrow] (-10.2,-10.7) -- (-0.2,-10.7);
\end{tikzpicture}
\caption{Algorithm scheme of the solution to (\ref{ex2-L})}
\end{figure}
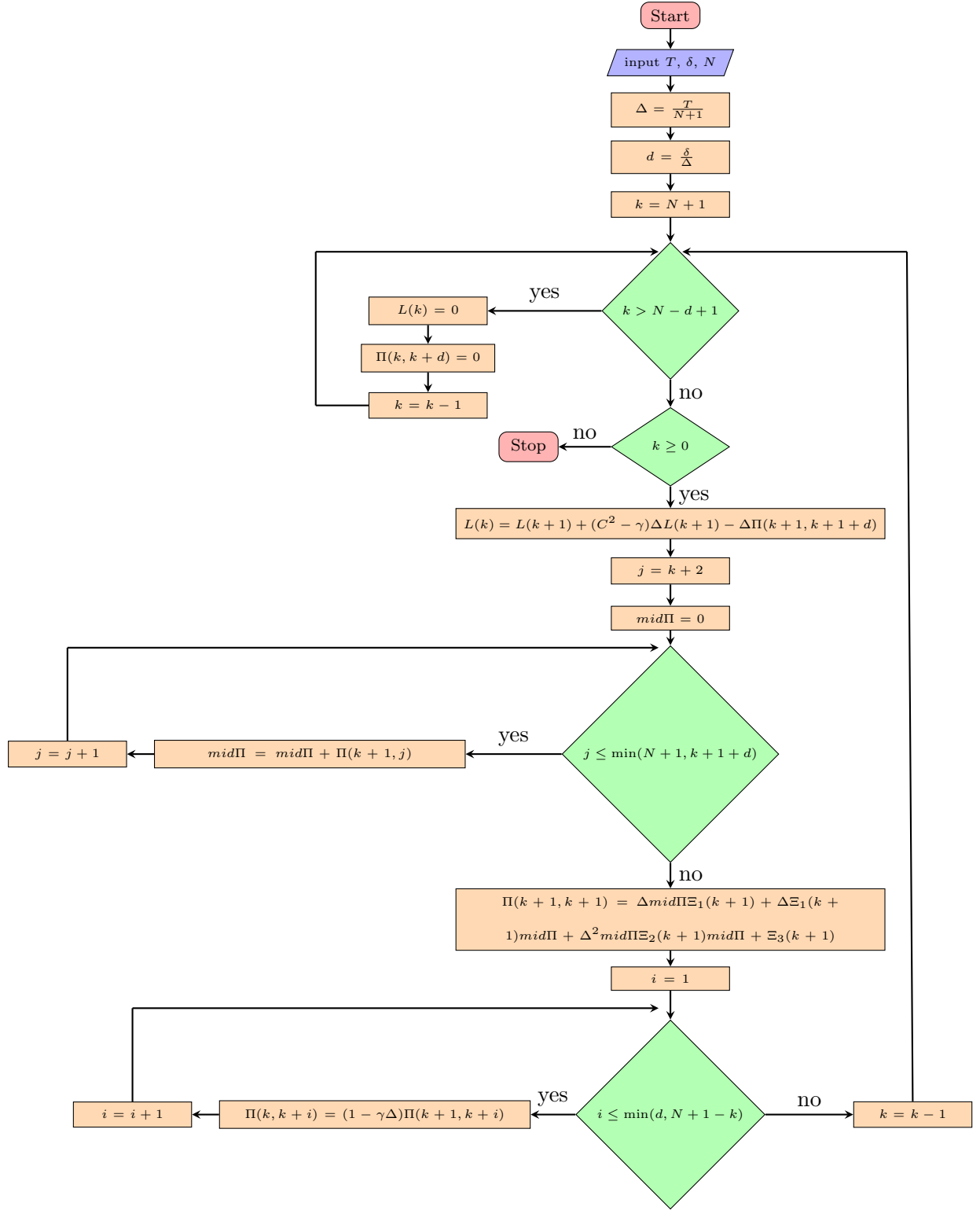

Let $\delta=1$, $\gamma=0.02$, $V=10$, $x_0=10$, $\alpha=2$, $C=-\frac{\sqrt\gamma}{2}$, $\bar{D}_1=-\bar{B}_1C^{-1}$, $\bar{D}_2=e^{\gamma\delta}$, $Q_1=\frac{V}{2x_0^2}$, $\bar{Q}_1=-\frac{e^{-\gamma\delta}V}{2x_0^2}$, $t_0=0$, $T=10$, $\varphi(0)=1$. Using the above algorithm scheme, we can plot the pictures of $P_1(\cdot), L_1(\cdot), L_2(\cdot), L_3(\cdot)$ as follows:
\begin{figure}[H]
  \centering

  \subfloat[The solution to linear ODE (\ref{ex2-P1})] 
    {
        \begin{minipage}[t]{0.5\textwidth}
            \centering          
            \includegraphics[width=\textwidth]{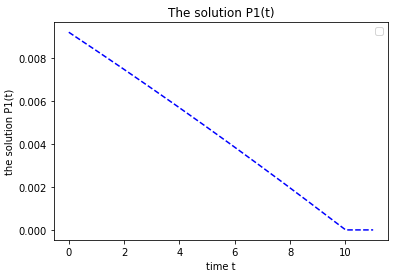}   
        \end{minipage}%
    }
    \subfloat[The solution to (\ref{ex2-L}) --- $L_1(t)$] 
    {
        \begin{minipage}[t]{0.5\textwidth}
            \centering      
            \includegraphics[width=\textwidth]{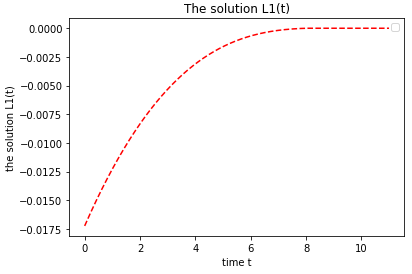}   
        \end{minipage}
    }%
\end{figure}

\begin{figure}[H]
  \centering

  \subfloat[The solution to (\ref{ex2-L}) --- $L_2(t)$] 
    {
        \begin{minipage}[t]{0.5\textwidth}
            \centering          
            \includegraphics[width=\textwidth]{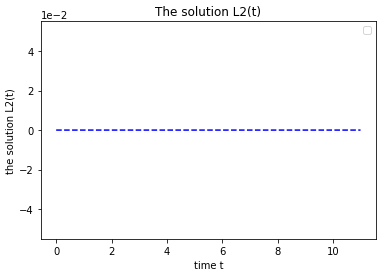}  
        \end{minipage}%
    }
    \subfloat[The solution to (\ref{ex2-L}) --- $L_3(t)$] 
    {
        \begin{minipage}[t]{0.5\textwidth}
            \centering      
            \includegraphics[width=\textwidth]{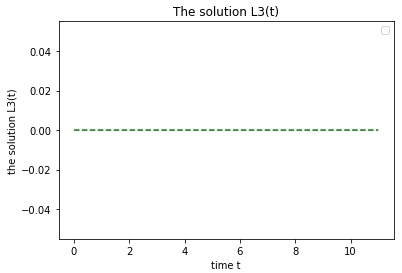}   
        \end{minipage}
    }%
\end{figure}

The simulation results of the optimal strategies $u_1(\cdot)$, $u_2(\cdot)$ are as follows:
\begin{figure}[H]
  \centering

  \subfloat[The optimal strategy $u_1(t)$] 
    {
        \begin{minipage}[t]{0.5\textwidth}
            \centering          
            \includegraphics[width=\textwidth]{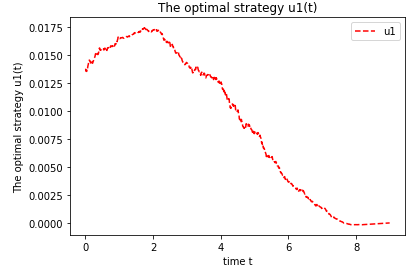}  
        \end{minipage}%
    }
    \subfloat[The optimal strategy $u_2(t)$] 
    {
        \begin{minipage}[t]{0.5\textwidth}
            \centering      
            \includegraphics[width=\textwidth]{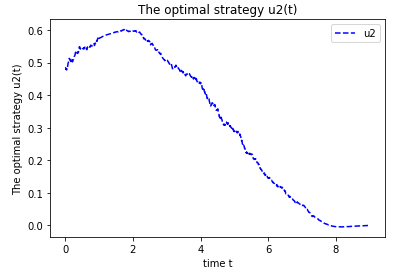}   
        \end{minipage}
    }%
\end{figure}

In the above two figures, the left/right figure draws the optimal strategy of the follower/leader. Judging from the trend of the two figures, the optimal strategy of the follower follows the leader, which is consistent with the essence of the leader-follower problem. Furthermore, the value of $u_2(\cdot)$ seems to be greater than $u_1(\cdot)$, this is rational because the leader wants to keep the technology ahead and so he should invest more money than the follower.

\section{Concluding remarks}

This paper is an extension of \cite{XSZ18}. Our model is quite general, in which the state equations of leader and follower both contain state delay and control delay, moreover, which both enter into the diffusion. First we deal with the optimization problem of the follower, which is an linear quadratic stochastic optimal control problem with time delay, for any choice of the leader. By introducing a Pseudo-Riccati equation (see (\ref{P1-equation})), we give the sufficient and necessary condition of the solvability for follower's problem (see Theorem \ref{thm3.2}). Next we address the optimization problem of the leader, which is an linear quadratic stochastic optimal control problem with a state equation formed by an SDDE and an ABSDE. We stack the forward variables and the backward variables obtained in the optimization of the follower and the leader, then introduce a special matrix equation (see (\ref{L-equation})), and finally we derive the necessary condition of the solvability for the leader's problem (see Theorem \ref{thm4.2}). In summary, the open-loop stackelberg strategy is shown in Theorem \ref{thm4.3}. At the end of this paper, we take two examples to illustrate the applications of the above theoretical results.

The assumptions ${\bf (A1)\mbox{-}(A3)}$ could be weakened, and some more general cases can be dealt with. Some reasonable relations among the adjoint variables and the state variable (including its delayed and time-advanced terms) should be proposed and investigated carefully. The solvability of (Pseudo-)Riccati equations are interesting. However, these topics are difficult and very challenging. We will consider them in the near future.


\end{document}